\documentclass[11pt,letterpaper]{article}

\usepackage{renstyles}
\usepackage{enumitem}

\newcommand{\KL}{\mathrm{KL}} 
\newcommand{\CaseI}{{\rm (Case~I)}} 
\newcommand{\CaseII}{{\rm (Case~II)}} 
 
\newcommand{\md}{\,\mathrm{d}}   
\usepackage{pgfplots}
\pgfplotsset{compat=1.18}

\title{Sampling Using Hybrid Stochastic Dynamics}
\author{
Bj\"orn Engquist\thanks{Department of Mathematics and the Oden Institute, The University of Texas at Austin, Austin, TX 78731; engquist@oden.utexas.edu}
\and
Kui Ren\thanks{Department of Applied Physics and Applied Mathematics, Columbia University, New York, NY 10027; kr2002@columbia.edu}
\and 
Yunan Yang\thanks{Department of Mathematics, Cornell University, Ithaca, NY 14850; yunan.yang@cornell.edu}
\thanks{This work is partially supported by the National Science Foundation through grants DMS-2208504 (BE), DMS-1937254 (KR), DMS-2309802 (KR), and DMS-2409855 (YY).  YY~was also partially supported by Office of Naval Research through grant N00014-24-1-2088.}
}

\date{}
\begin{document}


\maketitle



\begin{abstract}
This work proposes a framework for sampling from the Gibbs distribution of a given potential using hybrid stochastic dynamics. In this framework, two distinct sampling dynamics are run in different regions of the state space. The two dynamics are coupled across the interface through natural transmission conditions that preserve the target distribution. Using a specially constructed regularization scheme, we establish an exponential rate of convergence for the hybrid dynamics to equilibrium. We also analyze the metastability properties of the hybrid dynamics in a radially symmetric landscape, showing that the hybrid scheme can improve the mean exit time. This advantage is further confirmed by the numerical experiments.
\end{abstract}





\section{Introduction}
\label{SEC:Intro}

Sampling from Gibbs distributions is a fundamental task in scientific and statistical computing. Let $\Omega\subseteq\bbR^d$ ($d\ge 1$) be the given state space. Given a potential energy landscape $F:\Omega\to \bbR$ and a parameter $\varepsilon>0$, the objective is to sample from the Gibbs distribution
\begin{equation}\label{EQ:pi}
\pi(\bx)=\frac{1}{Z} e^{-F(\bx)/\eps},
\qquad
Z:=\int_\Omega e^{-F(\bx)/\eps}\,\md\bx.
\end{equation}
A standard approach is to simulate the overdamped Langevin dynamics
\begin{equation}\label{EQ:Langevin}
    dX_t=-\nabla F(X_t)\,dt+\sqrt{2\eps}\,dW_t,
\end{equation}
where $W_t$ is a standard $d$-dimensional Brownian motion. The invariant measure of~\eqref{EQ:Langevin} is precisely $\pi$. This dynamics forms the basis of many modern sampling algorithms; see, for instance,~\cite{BrGeJoMe-Book11,ChErLiShZh-FCM25,DuMo-AAP17,ErHo-PMLR21,LeRiGe-NIPS18,RoSt-MCAP02,VeWi-NIPS19,XiShLiByGi-SPL14} and references therein.

Despite its simplicity and theoretical appeal, overdamped Langevin dynamics often suffers from severe metastability in multi-well energy landscapes. That is, when the potential $F$ contains multiple local minima separated by high barriers, the process may remain trapped in one metastable basin for exponentially long times before transitioning to another. This phenomenon leads to slow mixing and poor sampling efficiency, especially in low-temperature (i.e., small $\eps$) regimes. To address this difficulty, various accelerated sampling strategies have been proposed~\cite{LeRoSt-Book10, Pavliotis-Book14,RoRo-Book04}, including underdamped dynamics~\cite{BoEbZi-AAP20,CaLuWa-ARMA23,ChChBaJo-PMLR18, DaRi-Bernoulli20,EbGuZi-AOP19,
LeMaSt-IMAJNA16,MaChJiFlJo-Bernoulli21,Monmarche-EJS21,RiVo-arXiv22}, preconditioned Langevin dynamics~\cite{FaSaSk-JCP14, GiCa-JRSS11, LaRoRo-AAP13,LiChCaCa-AAAI16,TiPa-JRSSB18}, adaptive diffusion methods~\cite{EnReYa-CAMC24,RiQuRiSc-SIAM24}, tempering, replica exchange and annealing-based methods~\cite{DuLiPlDo-MMS12,EaDe-PCCP05,GeLeRi-NeurIPS18, 
 MaPa-EL92, SuSyBoCa-NeurIPS22, SwWa-PRL86, SyBoDeDo-JRSSB22}, and many more~\cite{AnLi-AOS21,BeHeDoJa-arXiv19,BoVoDo-JASA18,HaScZh-Nonlinearity16,LiWa-NeurIPS16,LuSlWa-Nonlinearity23,ReWe-SIAMJUQ21,ZhCh-ICLR22}. A common principle underlying many successful approaches is to modify the local geometry of the dynamics in order to facilitate barrier crossing while preserving the target Gibbs law~\eqref{EQ:pi}.

In this work, we propose a framework for \emph{hybrid dynamics sampling}, in which two distinct diffusion dynamics are run in different regions of the state space. More precisely, the state space is partitioned into two subregions, for example, according to a level set of the potential, and different stochastic dynamics are prescribed in the two regions. The coefficients are chosen so that the resulting hybrid process still preserves the original Gibbs distribution globally. The motivation is to retain the favorable properties of overdamped Langevin dynamics in regions where the potential already provides efficient confinement, while replacing the dynamics in metastable regions by other dynamics whose effective geometry is better suited for rapid exploration.

Our primary example is a hybridization between the standard overdamped Langevin dynamics and the derivative-free adaptive diffusion introduced in~\cite{EnReYa-CAMC24,EnReYa-arXiv24}, although the framework developed here is more general. The earlier adaptive-variance sampler showed that, on bounded domains, a carefully chosen state-dependent diffusion coefficient can flatten the effective geometry inside nonconvex potential wells and thereby accelerate metastable transitions. However, using this derivative-free dynamics globally is not always appropriate: on unbounded domains, removing the Langevin drift everywhere may destroy the confining mechanism at infinity, and the resulting process need not converge to the target Gibbs distribution. The hybrid construction is also motivated by a simple numerical observation (see Section~\ref{subsec:regularized-hd-1d}).  If the derivative-free diffusion coefficient is used on the whole space, then the coefficient grows exponentially in high-energy regions, and the resulting numerical scheme can lose confinement immediately.

The hybrid dynamics proposed here localizes this acceleration mechanism. Inside a prescribed interior region, the drift is removed, and the diffusion coefficient is modified so that the Gibbs measure remains invariant; in particular, the weighted conductance density $a(\bx)\pi(\bx)$ becomes constant there. Outside this region, the dynamics revert to overdamped Langevin dynamics, preserving confinement and convergence. 

The main part of this work provides a rigorous convergence theory for the hybrid dynamics. Because the coefficients are discontinuous across the switching interface, the corresponding Fokker--Planck equation naturally takes the form of a transmission problem. We introduce a smooth regularization of the coefficients, establish uniform entropy dissipation estimates independent of the regularization parameter, and prove compactness and convergence of weak solutions to the limiting hybrid Fokker--Planck system. This yields exponential convergence of the hybrid dynamics toward the Gibbs measure in relative entropy under a logarithmic Sobolev inequality assumption.

We also analyze the metastability properties of the hybrid dynamics in a radial double-well landscape, as an example of situations where the hybrid dynamics achieve the better mixing speed of the two separate dynamics. Using the radial Poisson problem and Laplace asymptotics, we compare the mean transition times between metastable wells for the hybrid dynamics and for overdamped Langevin dynamics. Our analysis shows that, while the overdamped Langevin transition time is governed by the classical saddle barrier height, the hybrid dynamics replaces this barrier by the switching level of the interface. Consequently, if the switching interface is chosen below the saddle level, precisely when $F_r-F_g<H-F_u$ (with $F_g$, $F_u$, $H$, and $F_r$ the energies of the lower well, upper well, saddle, and switching interface, respectively), the hybrid dynamics achieves an exponentially faster transition rate than the overdamped Langevin dynamics.

The remainder of the paper is organized as follows. We introduce the hybrid dynamics and its regularization in Section~\ref{SEC:Hybrid-Regularized}. We then analyze the regularized dynamics uniformly in $\delta$ and then pass to the sharp-interface limit in Section~\ref{SEC:all_properties}. The key estimate is the entropy dissipation identity, which gives compactness and transfers the exponential convergence rate to the limiting hybrid dynamics. In Section~\ref{SEC:Exit}, we compare the metastable transition behavior of the hybrid and overdamped Langevin dynamics in a radial double-well potential to show that the hybrid dynamics accelerates metastable transition in this setting. Numerical simulations are presented in Section~\ref{SEC:Num} to verify that the hybrid dynamics indeed sample the target distribution.  Concluding remarks are offered in Section~\ref{SEC:Concl}.

\section{Hybrid dynamics and diffuse-interface regularization}
\label{SEC:Hybrid-Regularized}

We introduce the sharp-interface hybrid sampler and then define a smooth diffuse-interface approximation. The regularization is chosen so that the Gibbs density remains invariant for every $\delta>0$, while the coefficients remain uniformly controlled as $\delta\to0$.

\subsection{Sharp-interface hybrid dynamics}
\label{SUBSEC:Hybrid}

In our hybrid approach, we divide the state space $\Omega$ into two regions. Two different dynamics are run in the regions. We simultaneously consider two cases:
\begin{itemize}
    \item \hypertarget{case1}{$\CaseI$}: The state space $\Omega$ is a bounded domain, with boundary $\partial\Omega\in \cC^2$.
    \item \hypertarget{case2}{$\CaseII$}: The state space is $\Omega=\bbR^d$.
\end{itemize}

We assume that the potential function $F$ has the following properties.
\begin{itemize}
    \item[\hypertarget{F-a}{\bf (F-a)}] $F$ is sufficiently smooth, has a bounded gradient, and is bounded from below. More precisely:
\begin{equation}\label{EQ:F Assumptions}
    F\in\cC^2(\Omega),\qquad \nabla F\in L^\infty(\Omega), \qquad \mbox{and}\qquad \inf_\Omega F >-\infty\,.
\end{equation}
\end{itemize}
Note that the property $\nabla F\in L^\infty(\Omega)$ is necessary as it does not directly follow from the first property in \hyperlink{case2}{$\CaseII$}. 
We assume throughout that the normalization constant $$Z=\int_\Omega e^{-F(\bx)/\eps}\,\md\bx<\infty.$$ 

In this work, we use two related ways of partitioning the domain. 
The main analysis in~\Cref{SEC:Hybrid-Regularized,SEC:all_properties} treats the case where the switching interface is an entire level set of the potential.  For a given $F_0\in\mathbb{R}$, define
\[
    \Gamma:=\{\bx\in\Omega:F(\bx)=F_0\}.
\]
In the full-level-set setting, we assume that $\Gamma$ is sufficiently smooth, say $\Gamma\in\mathcal C^1$, and separates $\Omega$ into two regions
\[
    \Omega_- := \{\bx\in\Omega:F(\bx)<F_0\},
    \qquad
    \Omega_+ := \{\bx\in\Omega:F(\bx)>F_0\}.
\]
The whole level set $\Gamma$ is then used as the interface between the two dynamics. 

The exit-time analysis in~\Cref{SEC:Exit} treats a radially symmetric, possibly nonconvex case, where the relevant level set may have several connected components, and the switching interface is chosen to be the outermost one, namely the component adjacent to the exterior region. Until~\Cref{SEC:Exit}, we work exclusively in the full-level-set setting.
 
Our goal is to be able to run two different sampling schemes in $\Omega_-$ and $\Omega_+$ while preserving the target distribution in the whole space $\Omega$. More precisely, we run the overdamped Langevin dynamics~\eqref{EQ:Langevin} in $\Omega_+$, and a derivative-free adaptive diffusion dynamics in $\Omega_-$.
 
For a given $\eps$, we define the functions $a:\Omega\to \bbR$ and $b:\Omega\to \bbR$ as
\begin{equation}\label{EQ:Coefficients True}
    a(\bx)=
\begin{cases}
a_-:=\eps e^{(F(\bx)-F_0)/\eps} & \bx\in \Omega_-\\
a_+:=\eps, & \bx\in \Omega\setminus \Omega_-
\end{cases}\,,
\quad
b(\bx)=
\begin{cases}
b_-:=0, & \bx\in \Omega_-\\
b_+:=-\nabla F(\bx), & \bx\in \Omega\setminus\Omega_-
\end{cases}\,.  
\end{equation}
Note that we scale the diffusion coefficient inside $\Omega_-$ in a way such that $a(\bx)$ is continuous across $\Gamma$. The continuity of $a$ is an essential property to enforce; see Remark~\ref{REM:Continuity of a}.

The hybrid dynamics we are interested in are then
\begin{equation}\label{EQ:SDE HB}
\md X_t^\pm=b_\pm(X_t^\pm)\md t+\sqrt{2a_\pm(X_t^\pm)}\,\md W_t,\quad \mbox{in}\ \Omega_\pm\,.
\end{equation}
In \hyperlink{case1}{$\CaseI$}, that is, when $\Omega$ is bounded, the process in $\Omega_+$ is reflected at $\partial\Omega$ to enforce the no-flux condition at $\partial\Omega$.

The sampling dynamics in the inner region $\Omega_-$ is purely diffusive in nature, in the sense that the drift term is turned off (since $b_-=0$). To keep the stationary distribution Gibbsian, the diffusion coefficient is chosen to have the particular form of $a_-$. This leads to the fact that the weighted conductance density $a_-(\bx)\pi(\bx)$ is constant: $a_-(\bx)\pi(\bx) =
\eps Z^{-1}e^{-F_0/\eps}$. Thus, although the invariant measure remains Gibbsian, the effective conductivity governing the associated Dirichlet form becomes spatially uniform in $\Omega_-$; see more detailed discussions in~\Cref{SEC:Exit} and ~\cite{EnReYa-CAMC24}. Therefore, this dynamics removes the exponential weighting induced by the
potential inside the $\Omega_-$ and replaces it with a flat diffusion geometry. Hence, for a potential whose landscape below the level set $\Gamma$ is complex, the hybrid dynamics can potentially achieve better overall performance than running the overdamped Langevin dynamics in the whole state space $\Omega$; see numerical examples in Section~\ref{SEC:Num}.

The hybrid process $X_t$ is a Markov process on $\Omega$. Its infinitesimal generator is
\begin{equation}\label{EQ:Generator}
    \cL f(\bx) =a(\bx) \Delta f + b(\bx) \cdot \nabla f(\bx)\,.
\end{equation}
The domain of $\cL$, that is, which functions $f$ count as admissible test functions, requires that $f$ and $a \bn_{\Gamma_-}\cdot\nabla f$ be continuous across $\Gamma$. Therefore, the adjoint Fokker--Planck operator $\cL^*$ acts on densities $\rho$ with the dual transmission conditions. That is, the Fokker--Planck system for the hybrid dynamics is
\begin{equation}\label{EQ:FP HB}
\begin{array}{rcll}
\partial_t\rho_\pm
&=&
-\nabla\cdot(b_\pm \rho_\pm)+\Delta(a_\pm \rho_\pm), & \text{in }\Omega_\pm\\[1ex]
(a_-\rho_-)|_{\Gamma_-}  &=&
(a_+\rho_+)|_{\Gamma_+}, & \text{on}\ \Gamma\\
\bn_\Gamma\cdot J_-(\rho_-)|_{\Gamma_-}
&=&
\bn_\Gamma\cdot J_+(\rho_+)|_{\Gamma_+}, &\text{on}\ \Gamma
\end{array}
\end{equation}
where the flux
\begin{equation}\label{EQ:Flux}
    J_\pm(\rho_\pm):=b_\pm \rho_\pm-\nabla(a_\pm\rho_\pm)\,,
\end{equation}
and $\bn_\Gamma$ is the unit normal to $\Gamma$ (pointing, say, from $\Omega_-$ to $\Omega_+$). Since the coefficients $(a,b)$ are discontinuous across $\Gamma$, the limiting Fokker--Planck equation takes the form of a transmission problem. The matching conditions in~\eqref{EQ:FP HB} express conservation of the weighted conductance density and the probability flux across the switching interface. These are the natural transmission conditions for a conservative diffusion process in a heterogeneous environment. With the continuity of $a(\bx)$ across the interface $\Gamma$, the first interface condition implies the continuity of the density across the interface: $\rho|_{\Gamma_-}=\rho|_{\Gamma_+}$. 

When $\Omega$ is bounded, the no-flux boundary condition on the distribution level is described as
\begin{equation}\label{EQ:FP BC}
\bn_{\partial\Omega}\cdot J_+(\rho_+)|_{\partial\Omega}=0\,.
\end{equation}

We intentionally choose the interface to switch the dynamics as a level set of the potential rather than an arbitrary hypersurface. While this choice appears to be quite restrictive, it has some advantages. First, it ensures that the matching condition for continuity of $a$ can be enforced by a single scalar constant $\gamma:=\eps e^{-F_0/\eps}$, since $F$ is constant on $\Gamma$. Second, from the sampling perspective, the value of $F$ provides a natural energetic criterion for deciding where to modify the dynamics: the hybridization is activated in regions where the potential is below a prescribed threshold and metastability is expected to dominate.

\subsection{Diffuse-interface regularization} 
\label{SUBSEC:Hybrid-Regularized}

The hybrid sampling dynamics we proposed has the drift coefficient that is discontinuous across the switching interface $\Gamma$, and the associated Fokker--Planck equation~\eqref{EQ:FP HB} therefore takes the form of a transmission problem. This sharp-interface system is not very friendly for analyzing the performance of the system. 

Rather than working directly with the sharp-interface system, we approximate the interface by a thin transition layer of width $\cO(\delta)$, for some small $\delta>0$, across which the coefficients vary smoothly. The regularized dynamics can therefore be viewed as a diffuse-interface approximation of the hybrid process. As $\delta\to 0$, the transition layer collapses and the sharp-interface hybrid dynamics is formally recovered. To be precise, let $\mathrm{dist}(\bx,\Gamma)$ denote signed distance
(negative inside $\Omega_-$). For a given $\delta>0$, we denote
\[
    \chi_\delta(\bx)=\chi(\mathrm{dist}(\bx, \Gamma)/\delta)
\]
with a fixed smooth profile $\chi:\bbR\to[0,1]$ satisfying $\chi\equiv 1$ on $(-\infty,-1]$ and $\chi\equiv 0$ on $[1,\infty)$. See~\Cref{fig:chi} for an illustration of a candidate $\chi$.

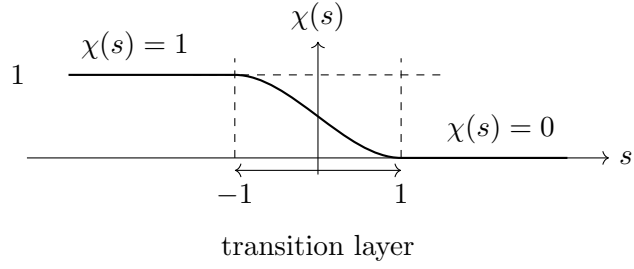
\begin{figure}[h]
\centering
\begin{tikzpicture}[scale=1.1]

\draw[->] (-3.5,0) -- (3.5,0) node[right] {$s$};
\draw[->] (0,-0.2) -- (0,1.4) node[above] {$\chi(s)$};

\draw[dashed] (-1.5,1) -- (1.5,1);
\node[left] at (-3.4,1) {$1$};

\draw[dashed] (-1,0) -- (-1,1.2);
\draw[dashed] (1,0) -- (1,1.2);
\node[below] at (-1,-0.2) {$-1$};
\node[below] at (1,-0.2) {$1$};
\draw[thick]
(-3,1)
.. controls (-2,1) and (-1.6667,1) ..
(-1,1)
.. controls (-0.3333,1) and (0.3333,0) ..
(1,0)
.. controls (1.6667,0) and (2,0) ..
(3,0);

\node[above] at (-2.2,1.05) {$\chi(s)=1$};
\node[above] at (2.2,0.05) {$\chi(s)=0$};

\draw[<->] (-1,-0.15) -- (1,-0.15);
\node[below] at (0,-0.8) {transition layer};

\end{tikzpicture}
\caption{An example of the smooth function $\chi(s)$.}\label{fig:chi}
\end{figure}

We then introduce the smoothed coefficients as
\begin{equation}\label{EQ:Coefficients Regularized}
\begin{array}{rcl}
a_\delta(\bx) &=& \chi_\delta(\bx)\,\eps e^{(F(\bx)-F_0)/\eps}+(1-\chi_\delta(\bx))\,\eps\,,\\[1ex]
b_\delta(\bx) &=& \nabla a_\delta(\bx)-\dfrac{a_\delta(\bx)}{\eps}\,\nabla F(\bx)\,.
\end{array} 
\end{equation}
Note that instead of directly regularizing both $a$ and $b$, our strategy here is to regularize the diffusion coefficient to $a_\delta$ and then select $b_\delta$ through the Gibbs-preserving relation. This is of key importance as it allows the dynamics to generate the same Gibbs distribution in the whole state space $\Omega$ for every $\delta>0$; see more discussions at the end of this section.

Formally, as $\delta\to0$,
\[
a_\delta(\bx)\to a(\bx),\qquad b_\delta(\bx)\to b(\bx)
\]
for every $\bx\in\Omega\setminus\Gamma$, and hence almost everywhere in $\Omega$.

\subsection{Basic properties of the regularized dynamics}

We first record the basic bounds on the regularized coefficients. These bounds are uniform in the regularization parameter and are the main reason the diffuse-interface approximation can be used to pass to the sharp-interface limit.

\begin{lemma}\label{LMMA:Uniform Coeff}
Assume that $F(\bx)\in \cC^2(\Omega)$ and the level set $\Gamma$ is at least $\cC^1$. Then the following holds.\\[1ex]
(i) The coefficient $a_\delta$ is globally uniformly bounded from below, that is, $\exists\ a_{\min}>0$, independent of $\delta$, such that
\begin{equation}\label{EQ:Uniform Lower Bound A}
    0<a_{\min}\le a_\delta(\bx),\ \ \forall\  \bx\in\Omega\,.
\end{equation}
(ii) In \hyperlink{case1}{$\CaseI$}, there exist constants $\Lambda$ and $M$, independent of $\delta$, such that
\begin{equation}\label{EQ:Uniform Bound Upper Bound AB}
0<a_{\min}\le a_\delta(\bx)\le \Lambda<\infty
\quad\forall \bx\in\Omega,
\qquad \|b_\delta\|_{L^\infty(\Omega)}\le M\,.
\end{equation}
(iii) In \hyperlink{case2}{$\CaseII$}, for any compact set $K\Subset\Omega$, there exists $\Lambda_K$ and $M_K$, independent of $\delta$, such that
\begin{equation}\label{EQ:Uniform Bound Upper Bound AB Local}
\|a_\delta\|_{L^\infty(K)}\le \Lambda_K,\qquad \|b_\delta\|_{L^\infty(K)}\le M_K\,.
\end{equation}
(iv) We have that
\[
    a_\delta(\bx)\to a(\bx),\quad b_\delta(\bx)\to b(\bx),\ \text{a.e.},\ \mbox{as}\ \delta\to 0\,,
\]
where $a(\bx)$ and $b(\bx)$ are defined in~\eqref{EQ:Coefficients True}.
\end{lemma}
\begin{proof} 
Let
\[
E(\bx):=e^{(F(\bx)-F_0)/\eps}.
\]
Then
\[
a_\delta(\bx)=\eps\bigl(1+\chi_\delta(\bx)(E(\bx)-1)\bigr)\,.
\]
Since $0\le \chi_\delta\le 1$, $a_\delta(\bx)$ is a convex combination of
$\eps E(\bx)$ and $\eps$.\\[1ex]
(i) The lower bound of $a_\delta$ follows directly from the assumption in~\eqref{EQ:F Assumptions} that $F$ is globally bounded from below. More precisely, we have that
$a_\delta\ge a_{\min}:=\eps\min\Bigl\{1,e^{(F_{\min}-F_0)/\eps}\Bigr\}$, where $F_{\min}:=\inf_\Omega F$.\\[1ex]
(ii) When $\Omega$ is bounded, $F$ is bounded from above on $\Omega$ by continuity. Let $F_{\max}:=\sup_\Omega F$. Then, we have that $a_\delta \le \Lambda$ with $\Lambda:=\eps\max\Bigl\{1,e^{(F_{\max}-F_0)/\eps}\Bigr\}$.

To bound $b_\delta$, we first observe, after some simple algebra, that 
\[
b_\delta
=
-(1-\chi_\delta)\nabla F
+\eps(E-1)\nabla\chi_\delta
\]
The first term is bounded by some $M_1$ by the assumption on $F$ in~\eqref{EQ:F Assumptions}:
\[
|(1-\chi_\delta)\nabla F|
\le |\nabla F|\le M_1
\]
For the second term, we note that
\[
\nabla\chi_\delta(\bx)
=
\frac1\delta
\chi'\!\left(\frac{\operatorname{dist}(\bx,\Gamma)}{\delta}\right)
\nabla\operatorname{dist}(\bx,\Gamma)\,.
\]
This gives that, for some $M_\chi>0$,
\[
|\nabla\chi_\delta(\bx)|\le \frac{M_\chi}{\delta},
\]
and $\nabla\chi_\delta$ is supported in the strip
\[
U_\delta:=\{\bx:\ |\operatorname{dist}(\bx,\Gamma)|<\delta\}.
\]
Because $F=F_0$ on $\Gamma$, we have that $E=1$ on $\Gamma$. Since $F\in \cC^2$,
$E$ is at least Lipschitz near $\Gamma$. Therefore, for $\bx\in U_\delta$, we have, for some Lipschitz constant $M_E$,
\[
|E(\bx)-1|
\le M_E\, |\operatorname{dist}(\bx,\Gamma)| 
\le M_E\delta.
\]
We now combine the above bounds to have
\[
\eps |E(\bx)-1|\,|\nabla\chi_\delta(\bx)|
\le
\eps(M_E \delta)\frac{M_\chi}{\delta} =:M_2.
\]
Taking $M:=M_1+M_2$ finishes the argument.\\[1ex]
(iii) follows from the same argument in (ii) on compact subsets in $\Omega=\bbR^d$, and (iv) follows directly from the definition of $a_\delta$ and $b_\delta$.
\end{proof}

The regularized dynamics is therefore
\begin{equation}\label{EQ:SDE Regularized}
\md X_t=b_\delta(X_t)\md t+\sqrt{2a_\delta(X_t)}\,\md W_t,
\end{equation}
and the associated Fokker--Planck equation is
\begin{equation}\label{EQ:FP Regularized}
\begin{array}{rcll}
\partial_t\rho_\delta
&=&
-\nabla\cdot(b_\delta\rho_\delta)+\Delta(a_\delta\rho_\delta), & \text{in }\Omega.
\end{array}
\end{equation}
We also define the regularized flux
\[
J_\delta(\rho_\delta):=b_\delta\rho_\delta-\nabla(a_\delta\rho_\delta)\,.
\]
The regularized coefficients $(a_\delta, b_\delta)$ are smooth and uniformly elliptic. Hence, the regularized Fokker--Planck equation becomes a classical uniformly parabolic drift-diffusion equation, allowing the use of standard entropy methods, compactness arguments, and parabolic regularity theory to analyze the system. We can then pass to the limit $\delta\to 0$ to understand the original sharp-interface transmission problem~\eqref{EQ:FP HB}.

When $\Omega$ is bounded, that is, in \hyperlink{case1}{$\CaseI$}, the no-flux boundary condition~\eqref{EQ:FP BC} is imposed for the regularized system:
\begin{equation}\label{EQ:FP Regularized BC}
J_\delta(\rho_\delta)\cdot \bn_{|\partial\Omega} = 0 \qquad \text{on }\partial\Omega.
\end{equation}
 
A key observation is that the Gibbs density $\pi$ defined in~\eqref{EQ:pi} is a stationary solution for the regularized system, independent of $\delta$, since:
\[
    J_\delta(\pi)=b_\delta\pi-\nabla(a_\delta\pi)\equiv 0,\quad \forall \delta>0.
\]
This happens because we choose a special way to regularize $a$ and adjust $b$ accordingly in~\eqref{EQ:Coefficients Regularized}. A naive independent regularization of $a$ and $b$ will change the Gibbs distribution inside the thin layer around the interface $\Gamma$.

\begin{remark}\label{REM:Continuity of a}
The enforcement of the continuity of $a(\bx)$ across $\Gamma$ is essential. Indeed, if $a(\bx)$ had a jump discontinuity across the interface, then the gradient of the regularization would scale like $\cO(\delta^{-1})$ inside the transition layer, producing unbounded drift terms in $b_\delta$. The continuity of $a(\bx)$ ensures that the regularized coefficients $(a_\delta, b_\delta)$ remain uniformly bounded as $\delta\to 0$, which is crucial for the compactness argument in the next sections for solutions to the regularized Fokker--Planck equation.
\end{remark}

From a computational perspective, the regularized dynamics may also be viewed as a practical implementation of the hybrid sampler. Indeed, any numerical discretization of a sharp switching interface introduces an effective smoothing at the mesh scale, so the parameter $\delta$ can also be interpreted as an interface-resolution parameter.

The goal of the next section is to establish that the hybrid dynamics preserves the target Gibbs distribution and converges to it at an exponential rate, uniformly with respect to the regularization parameter
$\delta$.

The main difficulty is that the limiting hybrid system~\eqref{EQ:FP HB} is a transmission problem with discontinuous coefficients across the interface $\Gamma$, for which standard parabolic theory does not directly apply. Our strategy is therefore to analyze the regularized system~\eqref{EQ:FP Regularized} and then pass to the limit $\delta\to 0$.

For each $\delta>0$, the regularized coefficients $(a_\delta,b_\delta)$ are smooth and uniformly elliptic (Lemma~\ref{LMMA:Uniform Coeff}), so the regularized Fokker--Planck equation~\eqref{EQ:FP Regularized} is a classical uniformly parabolic drift-diffusion equation. It admits smooth classical solutions with smooth initial data~\cite{LaSoUr-Book68}. In the rest of this work, we assume that the initial data we consider, $\rho_{0}$, satisfies
\begin{equation}
\label{EQ:IC ASS}
\begin{array}{lllll}
    \mbox{\hyperlink{case1}{$\CaseI$}:} & \rho_0\in C^\infty(\overline{\Omega}), & \rho_0>0, &
\int_\Omega\rho_0\,\md\bx=1, & \KL(\rho_0\,\|\,\pi)<+\infty\\[1ex]
    \mbox{\hyperlink{case2}{$\CaseII$}:} & \rho_0\in\cS(\bbR^d),& \rho_0>0,& \int_{\bbR^d}\rho_0\,\md\bx=1,& \KL(\rho_0\,\|\,\pi)<+\infty
\end{array}
\end{equation}
where $\cS(\bbR^d)$ is the standard Schwartz space of rapidly decreasing functions on $\bbR^d$.

We recall this classical result here for convenience.
\begin{theorem}[\cite{LaSoUr-Book68}]
\label{THM:FP Regularized Solution Theory}
Fix $\delta > 0$ and $T > 0$, and assume that $F$ satisfies \hyperlink{F-a}{\bf (F-a)} and that $\rho_0$ satisfies \eqref{EQ:IC ASS}. Then, the regularized Fokker--Planck equation~\eqref{EQ:FP Regularized} (with the no-flux boundary condition \eqref{EQ:FP Regularized BC} in \hyperlink{case1}{$\CaseI$}) admits a unique classical solution
\[
  \rho_\delta \in C^\infty\bigl([0,T] \times \overline{\Omega}\bigr)
\]
in \textup{\hyperlink{case1}{$\CaseI$}}, respectively
\[
  \rho_\delta \in C^\infty\bigl([0,T] \times \mathbb{R}^d\bigr),
  \qquad \text{with } \rho_\delta(t,\cdot) \in \mathcal{S}(\mathbb{R}^d)
  \text{ for every } t \in [0,T],
\]
in \textup{\hyperlink{case2}{$\CaseII$}}. Moreover, $\rho_\delta(t, \bx) > 0$ for every $(t, \bx) \in (0, T] \times \Omega$ and
\begin{equation}\label{EQ:Mass Conservation}
    \int_\Omega \rho_\delta(t,\bx)\,\md\bx=1,\quad \forall t\in [0, T]\,.
\end{equation}
\end{theorem}
\begin{proof}
By Lemma~\ref{LMMA:Uniform Coeff}, the coefficients $(a_\delta, b_\delta)$ are smooth and uniformly elliptic, with $a_\delta \geq a_{\min} > 0$, so \eqref{EQ:FP Regularized} is a linear uniformly parabolic equation with smooth coefficients. Existence, uniqueness, and $C^\infty$ regularity up to the boundary/initial time follow from classical parabolic theory: in \textup{\hyperlink{case1}{$\CaseI$}} this is standard for smooth initial data compatible with the no-flux condition~\cite[Ch.~IV--V]{LaSoUr-Book68}. In \textup{\hyperlink{case2}{$\CaseII$}} the same holds on $\bbR^d$. In divergence form, \eqref{EQ:FP Regularized} has bounded drift $\tfrac{a_\delta}{\eps}\nabla F$, so Aronson's Gaussian upper bounds for the fundamental solution and its derivatives~\cite{Ar-BAMS-67,Friedman-Book08} propagate the Schwartz decay of $\rho_0$ to $\rho_\delta(t,\cdot)$. In particular, $\rho_\delta(t,\cdot)$ and the flux $J_\delta(\rho_\delta)(t,\cdot)$ decay faster than any polynomial as $|\bx|\to\infty$, which is all that is needed for the integrations by parts at infinity used here and below. Integrating~\eqref{EQ:FP Regularized} over $\Omega$ and using the divergence theorem gives
\[
  \frac{d}{dt} \int_\Omega \rho_\delta(t, \bx)\, \md\bx
    = -\int_{\partial \Omega} \bn_{|\partial\Omega}\cdot J_\delta(\rho_\delta)\, dS = 0,
\]
where the boundary integral vanishes by the no-flux condition~\eqref{EQ:FP Regularized BC} in \hyperlink{case1}{$\CaseI$} and by Schwartz decay in \hyperlink{case2}{$\CaseII$}.~\eqref{EQ:Mass Conservation} then follows from the assumption that $\int_\Omega \rho_0\, \md\bx = 1$.
\end{proof}

\section{Entropy estimates, sharp-interface limit, and exponential convergence}\label{SEC:all_properties}

In this section, we analyze the regularized dynamics uniformly in $\delta$ and then pass to the sharp-interface limit. The key estimate is the entropy dissipation identity, which gives compactness and transfers the exponential convergence rate to the limiting hybrid dynamics.

\subsection{Uniform entropy dissipation}
\label{SUBSEC:Entropy Bound}

For convenience, we first rewrite~\eqref{EQ:FP Regularized} into the standard divergence (Onsager) form
\begin{equation}\label{EQ:FP Onsager Form}
\partial_t\rho_\delta
=
\nabla\cdot\left(a_\delta(\bx)\,\rho_\delta\,\nabla\log\frac{\rho_\delta}{\pi}\right)
\qquad \text{in }\Omega,
\end{equation}
using the identity $\nabla(a_\delta\rho) - b_\delta\rho = a_\delta\rho\,\nabla\log\!\left(\frac{\rho}{\pi}\right)$. The boundary condition~\eqref{EQ:FP Regularized BC}, required when $\Omega$ is bounded, that is, in \hyperlink{case1}{$\CaseI$}, then takes the form 
\begin{equation}\label{EQ:FP Onsager Form BC}
\left(a_\delta\rho_\delta\,\nabla\log(\rho_\delta/\pi)\right)\cdot \bn_{|\partial\Omega}=0, \qquad \mbox{on}\ \partial\Omega\,.
\end{equation}
We define the standard relative entropy energy
\begin{equation}\label{EQ:KL}
    \cE(\rho):=\KL(\rho\|\pi)=\int_\Omega \rho\log\frac{\rho}{\pi}\,\md \bx.
\end{equation}
Then the PDE~\eqref{EQ:FP Onsager Form} can be interpreted as a gradient flow of the entropy energy $\cE(\rho_\delta)$ in a Wasserstein-type metric $\mathsf W_\delta$ with position-dependent mobility $m_\delta(\bx,\rho):=a_\delta \rho$~\cite{EnReYa-arXiv24}. That is,
\begin{equation}\label{eq:GF_metric}
\partial_t\rho_\delta = \nabla_{\mathsf W_\delta}\cE(\rho_\delta)
\quad\Longleftrightarrow\quad
\partial_t\rho_\delta=\nabla\cdot \!\left(a_\delta\rho_\delta\nabla\frac{\delta \cE}{\delta \rho_\delta}\right)
=
\nabla\cdot \!\left(a_\delta\rho_\delta\,\nabla \log\frac{\rho_\delta}{\pi}\right)\,.
\end{equation}
Here $\mathsf W_\delta$ is the metric whose tangent vectors $\sigma=-\nabla\cdot(a_\delta\rho\nabla\phi)$ carry the weighted kinetic energy $\int_\Omega a_\delta\rho|\nabla\phi|^2\,\md\bx$; since $\delta\cE/\delta\rho=\log(\rho/\pi)+1$, the steepest descent of $\cE$ is precisely the Onsager form~\eqref{EQ:FP Onsager Form}. When $a_\delta$ is constant, $\mathsf W_\delta$ reduces, up to a time rescaling, to the usual quadratic Wasserstein metric~\cite{JoKiOt-SIAM98,AmGiSa-Book08}.

In the rest of the work, we assume that the potential $F$ is such that the stationary distribution $\pi$ in~\eqref{EQ:pi} satisfies the following properties:
\begin{itemize}
    \item [\hypertarget{pi-a}{($\pi$-a)}] (Log-Sobolev Inequality (LSI)) For all sufficiently smooth $\rho$ with $\dint_\Omega\rho\, \md\bx=1$ and $\rho\ll \pi$, we have
\begin{equation}\label{EQ:LSI}
\KL(\rho\|\pi)\le \frac{1}{2\lambda_{\mathrm{LSI}}}\int_\Omega \rho\left |
\nabla\log\frac{\rho}{\pi}\right |^2 d \bx,
\end{equation}
for some LSI constant $\lambda_{\mathrm{LSI}}>0$.
\item[\hypertarget{pi-b}{($\pi$-b)}] (Exponential Moment Condition) In \hyperlink{case2}{$\CaseII$}, we have that
\begin{equation}\label{EQ:Moment Cond}
    \dint_{\Omega} e^{\alpha|\bx|^2}\pi(\bx)\md \bx<\infty,\quad \mbox{for some}\ \alpha>0\,.
\end{equation}
\end{itemize}
Log-Sobolev inequalities for $\pi$, such as~\eqref{EQ:LSI}, are available under various scenarios. On bounded domains, LSI for $\pi$ is often available under mild regularity, though the LSI constants may depend on $\Omega$, $F$, and $\eps$~\cite{BaGeLe-Book14, HoSt-JSP87, Ledoux-SP04}.

The exponential moment condition~\eqref{EQ:Moment Cond} ensures that when $\KL(\rho\|\pi)$ is finite, $\rho$ has finite second moment. This is classical~\cite{AmGiSa-Book08,DuEl-Book97}. We reproduce it in the following lemma.
\begin{lemma}\label{LMMA:Moment}
In \hyperlink{case2}{$\CaseII$}, assume that $\pi$ satisfies the assumption \hyperlink{pi-b}{\rm ($\pi$-b)}. Then, for any sufficiently regular $\rho$ such that $\rho\ge0$, $\int_\Omega \rho \md\bx =1$, and $\KL(\rho\|\pi)\le M<+\infty$ for some constant $M$, there exists constant $c=c(M,\alpha,\pi)$ such that
\[
\int_{\Omega}|\bx|^2\rho(\bx)\md \bx \le c\,.
\]
\end{lemma}
\begin{proof}
The statement is trivial in \hyperlink{case1}{$\CaseI$}. In \hyperlink{case2}{$\CaseII$}, this is the classical Donsker--Varadhan/Gibbs variational bound~\cite{DuEl-Book97,AmGiSa-Book08}: taking $\Phi(\bx)=\theta|\bx|^2$ with $\theta\in(0,\alpha)$ in $\KL(\rho\|\pi)=\sup_\Phi\{\int\Phi\rho\,\md\bx-\log\int e^\Phi\pi\,\md\bx\}$ gives $\theta\int|\bx|^2\rho\,\md\bx\le\KL(\rho\|\pi)+\log\int e^{\theta|\bx|^2}\pi\,\md\bx$, and the last term is finite by~\hyperlink{pi-b}{($\pi$-b)}.
\end{proof}

We define the entropy dissipation functional (or the weighted Fisher information)
\[
\mathcal I_{a_\delta}(\rho_\delta\|\pi):=\int_\Omega a_\delta(\bx)\,\rho_\delta(\bx)\left|\nabla\log\frac{\rho_\delta}{\pi}\right|^2\md \bx\,.
\]
The following result is standard; see, for instance, ~\cite{ArMaToUn-CPDE01,BaEm-SP85,BaGeLe-Book14,CaJuMaToUn-MM01,MaVi-MC00}. We recall it for convenience.
\begin{lemma}\label{LMMA:KL Decay}
Assume that $\pi$ satisfies the LSI~\eqref{EQ:LSI}. Let $\rho_\delta$ be a smooth and strictly positive solution to~\eqref{EQ:FP Onsager Form} with initial condition $\rho_{0}$ such that $\dint_{\Omega}\rho_{0}\,\md\bx=1$ and $\KL(\rho_{0}\|\pi)<+\infty$. Assume further that $\rho_\delta$ satisfies the boundary condition~\eqref{EQ:FP Onsager Form BC} in \hyperlink{case1}{$\CaseI$} and decays sufficiently fast as $|\bx|\to\infty$ in \hyperlink{case2}{$\CaseII$}. 
Then for every $\delta>0$ and $t\ge0$, $\rho_\delta$ satisfies
\begin{equation}\label{EQ:KL Dis}
\frac{\md}{\md t}\KL(\rho_\delta(t)\|\pi)
=
-\mathcal I_{a_\delta}(\rho_\delta\|\pi) \le 0
\end{equation}
leading to
\begin{eqnarray}
\label{EQ:KL Decay}
\KL(\rho_\delta(t)\|\pi)
&\le & e^{-2a_{\min}\lambda_{\mathrm{LSI}}t}\,\KL(\rho_{\delta}(0)\|\pi),\\[1ex]
\label{EQ:Fisher Bound}
\int_0^T \mathcal I_{a_\delta}(\rho_\delta(s)\|\pi)\md s
&\le& \KL(\rho_{\delta}(0)\|\pi),
\qquad \forall T>0.
\end{eqnarray}
In particular, we have that, for any $0<t_1<t_2\le T$,
\begin{equation}\label{EQ:Fisher Bound B}
\int_{t_1}^{t_2}\!\!\int_\Omega a_\delta \left|\nabla\log\!\left(\frac{\rho_\delta}{\pi}\right)\right|^2 \rho_\delta\,\md \bx\,\md t
\le \KL(\rho_\delta(t_1)\|\pi).
\end{equation} 
\end{lemma}

\begin{proof}
Differentiating $\KL(\rho_\delta\|\pi)$ along~\eqref{EQ:FP Onsager Form}, using mass conservation~\eqref{EQ:Mass Conservation}, and integrating by parts with the boundary condition~\eqref{EQ:FP Onsager Form BC} in \hyperlink{case1}{$\CaseI$} (or the decay condition in \hyperlink{case2}{$\CaseII$}) gives the dissipation identity~\eqref{EQ:KL Dis}. This is the standard entropy-entropy-dissipation computation~\cite{ArMaToUn-CPDE01,BaEm-SP85}. Since $a_\delta\ge a_{\min}$, we have $\cI_{a_\delta}(\rho_\delta\|\pi)\ge a_{\min}\int_\Omega\rho_\delta|\nabla\log(\rho_\delta/\pi)|^2\,\md\bx$, so the LSI~\eqref{EQ:LSI} yields $\frac{\md}{\md t}\KL(\rho_\delta(t)\|\pi)\le-2a_{\min}\lambda_{\mathrm{LSI}}\KL(\rho_\delta(t)\|\pi)$, and~\eqref{EQ:KL Decay} follows by Gr\"onwall's inequality. Integrating the identity~\eqref{EQ:KL Dis} in time and discarding the nonnegative terminal entropy gives~\eqref{EQ:Fisher Bound} and, over $[t_1,t_2]$,~\eqref{EQ:Fisher Bound B}.
\end{proof}
 
The key point of Lemma~\ref{LMMA:KL Decay} is that, thanks to the uniform ellipticity $a_\delta\ge a_{\min}>0$, the dissipation $\cI_{a_\delta}$ controls the (weighted) Fisher information uniformly in $\delta$, so the LSI yields a $\delta$-independent entropy decay rate. The two resulting uniform bounds, on the Fisher information and on its time integral~\eqref{EQ:Fisher Bound B}, furnish, respectively, spatial and (weak) time regularity of $\rho_\delta$. In~\Cref{SUBSEC:Compactness Plus Limit} we turn these into compactness of $\{\rho_\delta\}$, from which we extract a convergent subsequence and pass to the limit $\delta\to0$.

\subsection{Passage to sharp-interface limit}
\label{SUBSEC:Compactness Plus Limit}

The goal of this section is to pass to the limit $\delta \to 0$ in the regularized Fokker--Planck equation and recover a weak solution of the hybrid transmission problem introduced in Section~\ref{SEC:Hybrid-Regularized}. We follow the standard strategy of using the compactness provided by the entropy dissipation estimates to extract a convergent subsequence and then identifying the limit by passing to the weak formulation~\cite{Jungel-Book16}.

First, we observe that the assumption about the initial condition in~\eqref{EQ:IC ASS} immediately yields the following bounds, which we will use throughout this section.
\begin{lemma}\label{LMMA:KL-Fisher Bounds}
Under the assumptions~\eqref{EQ:F Assumptions} and~\eqref{EQ:IC ASS}, we have that
\begin{equation}\label{EQ:KL Bound by IC}
\sup_{t\geq 0}\KL\bigl(\rho_\delta(t)\,\|\,\pi\bigr)
\leq\KL(\rho_0\,\|\,\pi)<+\infty,
\end{equation}
uniformly in $\delta$, and
\begin{equation}
\label{EQ:Fisher Bound B by IC}
\sup_{\delta>0}\int_0^T\!\!\int_\Omega\rho_\delta(t,\bx)
\left|\nabla\log\frac{\rho_\delta(t,\bx)}{\pi(\bx)}\right|^2\,\md\bx\,dt
\leq\frac{1}{a_{\min}}\KL(\rho_0\,\|\,\pi)<+\infty.
\end{equation}   
\end{lemma}
\begin{proof}
The bound~\eqref{EQ:KL Bound by IC} follows directly from~\eqref{EQ:KL Dis} of Lemma~\ref{LMMA:KL Decay}. We then combine ~\eqref{EQ:Fisher Bound B} with the fact that $a_\delta\ge a_{\min}$ (given in Lemma~\ref{LMMA:Uniform Coeff}) to conclude~\eqref{EQ:Fisher Bound B by IC}.
\end{proof}

\subsubsection{Compactness of solution sequences}

The uniform entropy dissipation we see gives two key bounds on the interval $[0,T]$. The first is that $\sqrt{\rho_\delta}$ is bounded in $L^2((0,T);H^1_{\mathrm{loc}})$, and the other is that $\partial_t\rho_\delta$ is bounded in $L^2((0,T);H^{-1}_{\mathrm{loc}})$. We establish these in the following two lemmas. The strong local compactness of $\rho_\delta$ for positive times will allow us to take $\delta\to 0$ using the lower-semicontinuity of $\KL$ divergence.

We record a $\delta$-uniform local sup bound on $\rho_\delta$, used repeatedly below. The point is that this bound depends only on the ellipticity ratio and the drift bound, hence stays uniform as $\delta\to0$, even though the classical (Schauder) regularity constants degenerate since $\nabla^2 a_\delta=\cO(\delta^{-2})$.
\begin{lemma}\label{LMMA:Linfty}
Assume $F$ satisfies~\hyperlink{F-a}{\bf (F-a)} and $\rho_0$ satisfies~\eqref{EQ:IC ASS}. Then, for every compact $K\Subset\Omega$ (with $K=\Omega$ in \hyperlink{case1}{$\CaseI$}) and every $T>0$, there is a constant $C^{(1)}_{T,K}>0$, independent of $\delta$, such that
\begin{equation}\label{EQ:Linfty}
\|\rho_\delta\|_{L^\infty((0,T)\times K)}\le C^{(1)}_{T,K}\,.
\end{equation}
\end{lemma}
\begin{proof}
Using $b_\delta=\nabla a_\delta-\dfrac{a_\delta}{\eps}\nabla F$, equation~\eqref{EQ:FP Regularized} can be written in divergence form as
\[
\partial_t\rho_\delta=\nabla\cdot\!\Big(a_\delta\nabla\rho_\delta+\tfrac{a_\delta}{\eps}\rho_\delta\nabla F\Big).
\]
By Lemma~\ref{LMMA:Uniform Coeff}, the principal coefficient satisfies $a_{\min}\le a_\delta\le\Lambda_K$ on $K$ and the first-order coefficient obeys $|\frac{a_\delta}{\eps}\nabla F|\le\frac{\Lambda_K}{\eps}\|\nabla F\|_{L^\infty(K)}$, both uniformly in $\delta$. The local boundedness estimate of De~Giorgi--Nash--Moser for divergence-form parabolic equations~\cite[Ch.~III]{LaSoUr-Book68} (equivalently, Aronson's Gaussian upper bound for the fundamental solution~\cite{Ar-BAMS-67}) then yields a sup bound on $\rho_\delta$ over parabolic cylinders in $(0,T]\times K$ whose constant depends only on $d$, the ellipticity ratio $\Lambda_K/a_{\min}$, the first-order bound, $K$, and $T$, hence is independent of $\delta$. Near $t=0$ the bound is controlled by $\|\rho_0\|_{L^\infty(\Omega)}$ through the parabolic maximum principle, again with $\delta$-independent constants. Together with $\|\rho_\delta(t)\|_{L^1}=1$ given in Theorem~\ref{THM:FP Regularized Solution Theory}, this gives~\eqref{EQ:Linfty} on all of $[0,T]\times K$.
\end{proof}

We first establish the following $H^1$ control of $\sqrt{\rho_\delta}$ for $t>0$.
\begin{lemma}\label{LMMA:Grad Bound}
Assume that $F$ satisfies assumption~\hyperlink{F-a}{\bf (F-a)} and $\rho_{0}$ satisfies ~\eqref{EQ:IC ASS}.
Then, for every compact $K\Subset\Omega$
(with $K=\Omega$ in \hyperlink{case1}{$\CaseI$}), there exists a constant $C_{T,K}>0$, independent of $\delta$, such that
\begin{equation}
\sup_{\delta>0}\int_0^T\!\!\int_K
\bigl|\nabla\sqrt{\rho_\delta(t,\bx)}\bigr|^2\,\md\bx\,dt
\leq C_{T,K}.
\end{equation}
In particular, $\{\sqrt{\rho_\delta}\}_{\delta>0}$ is bounded in
$L^2((0,T);H^1(K))$ uniformly in $\delta$.
\end{lemma}
\begin{proof}
By the assumptions, we have that $\rho_\delta$ is smooth and strictly positive. Therefore, we have that $\sqrt{\rho_\delta}$ is smooth and the pointwise
identity
\[
\bigl|\nabla\sqrt{\rho_\delta}\bigr|^2
=\frac{|\nabla\rho_\delta|^2}{4\rho_\delta}
=\frac{\rho_\delta|\nabla\log\rho_\delta|^2}{4}
\]
holds. Using 
\[
\nabla\log\!\left(\frac{\rho_\delta}{\pi}\right)=\nabla\log\rho_\delta-\nabla\log\pi\,.
\]
and the inequality $u^2\le 2(u-v)^2 + 2v^2$, we conclude that
\[
\left|\nabla\log\rho_\delta\right|^2
\le 2\left|\nabla\log\!\left(\frac{\rho_\delta}{\pi}\right)\right|^2 + 2|\nabla\log\pi|^2.
\]
Therefore, we have that
\[
\int_0^T\!\!\int_K\bigl|\nabla\sqrt{\rho_\delta}\bigr|^2\,\md\bx\,dt
\leq
\frac{1}{2}\int_0^T\!\!\int_K \rho_\delta
\left|\nabla\log\frac{\rho_\delta}{\pi}\right|^2\,\md\bx\,dt
+\frac{1}{2}\int_0^T\!\!\int_K\rho_\delta|\nabla\log\pi|^2\,\md\bx\,dt.
\]
By~\eqref{EQ:Fisher Bound B by IC}, the first term is uniformly bounded. By~\hyperlink{F-a}{\bf (F-a)}, $\nabla\log\pi=-\varepsilon^{-1}\nabla F\in L^\infty(K)$, and by~\eqref{EQ:Mass Conservation}, $\int_K\rho_\delta(t)\leq 1$, so the second term is bounded by $\frac{T}{2}\|\nabla\log\pi\|_{L^\infty(K)}^2$.
\end{proof}

We now sharpen Lemma~\ref{LMMA:Grad Bound} by establishing a uniform $H^1$ bound on the product $a_\delta \rho_\delta$. This bound will be the key ingredient in deducing the continuity of $\rho$ across the interface $\Gamma$ in the sharp-interface limit; see Theorem~\ref{THM:Continuity of Rho}.

\begin{lemma}\label{LMMA:H1-aRho}
Assume that $F$ satisfies assumption~\hyperlink{F-a}{\bf (F-a)} and $\rho_0$
satisfies~\eqref{EQ:IC ASS}. Then, for every compact
$K \Subset \Omega$ (with $K = \Omega$ in \hyperlink{case1}{$\CaseI$}), there
exists a constant $\widetilde{C}_{T,K} > 0$, independent of $\delta$,
such that
\begin{equation}
\label{EQ:H1-aRho}
\sup_{\delta > 0}\int_0^T \int_K
   \bigl|\nabla(a_\delta \rho_\delta)\bigr|^2 \, \md\bx \, dt
\;\leq\; \widetilde{C}_{T,K}.
\end{equation}
In particular, $\{a_\delta \rho_\delta\}_{\delta > 0}$ is bounded in
$L^2\bigl((0,T); H^1(K)\bigr)$ uniformly in $\delta$.
\end{lemma}
\begin{proof}
Fix a compact set $K \Subset \Omega$ (with $K = \Omega$ in~\hyperlink{case1}{$\CaseI$}). Since $\rho_\delta$ is smooth and strictly positive, we may write
\begin{equation}\label{EQ:Grad-aRho-Split}
\nabla(a_\delta \rho_\delta) \;=\; a_\delta \nabla \rho_\delta
   + \rho_\delta \nabla a_\delta,
\end{equation}
and we estimate the two terms separately.

To estimate $a_\delta \nabla \rho_\delta$, we use $\nabla \log(\rho_\delta/\pi) = \nabla \rho_\delta / \rho_\delta
- \nabla \log \pi$ to get
\[
\nabla \rho_\delta \;=\; \rho_\delta \nabla \log\!\frac{\rho_\delta}{\pi}
   + \rho_\delta \nabla \log \pi.
\]
Therefore, by the inequality $|u+v|^2 \leq 2|u|^2 + 2|v|^2$, we have
\begin{equation}
\label{EQ:Grad-Rho-Split}
|a_\delta \nabla \rho_\delta|^2
\;\leq\; 2 a_\delta^2 \rho_\delta^2
   \Bigl|\nabla \log\tfrac{\rho_\delta}{\pi}\Bigr|^2
   + 2 a_\delta^2 \rho_\delta^2 |\nabla \log \pi|^2.
\end{equation}
By Lemma~\ref{LMMA:Uniform Coeff}, $a_\delta \leq \Lambda_K$ uniformly
on $K$ (with $\Lambda$ instead of $\Lambda_K$ in~\hyperlink{case1}{$\CaseI$}). By \BLUE{Lemma~\ref{LMMA:Linfty}}, there exists a constant $C_{T,K}^{(1)}$, independent of $\delta$, such that $\|\rho_\delta\|_{L^\infty((0,T) \times K)} \leq C_{T,K}^{(1)}$. Therefore, we have
\begin{equation}\label{EQ:Grad-Rho-Split-2}
|a_\delta \nabla \rho_\delta|^2
\;\leq\; 2 \Lambda_K^2 C_{T,K}^{(1)} \Bigl|\nabla \log\tfrac{\rho_\delta}{\pi}\Bigr|^2 \rho_\delta
   + 2 \Lambda_K^2 C_{T,K}^{(1)} |\nabla \log \pi|^2\rho_\delta.
\end{equation}
Integrating~\eqref{EQ:Grad-Rho-Split-2} on $(0,T) \times K$ and using the uniform Fisher-information control~\eqref{EQ:Fisher Bound B by IC}, we obtain
\begin{equation}
\label{EQ:Bound T1}
\int_0^T\!\!\int_K |a_\delta \nabla \rho_\delta|^2 \, \md\bx \, dt
\;\leq\; \frac{2 \Lambda_K^2 C_{T,K}^{(1)}}{a_{\min}} \cdot \mathrm{KL}(\rho_0\|\pi)
   + 2 \Lambda_K^2 T C_{T,K}^{(1)} \|\nabla \log \pi\|_{L^\infty(K)}^2,
\end{equation}
where we have also used the fact that $\rho_\delta \in L^\infty\bigl((0,T); L^1(\Omega)\bigr)$
with $\|\rho_\delta(t)\|_{L^1} = 1$ for all $t \in [0,T]$ given by Theorem~\ref{THM:FP Regularized Solution Theory} and the assumption~\eqref{EQ:IC ASS}
on $\rho_0$.

To estimate $\rho_\delta \nabla a_\delta$, we first recall from~\eqref{EQ:Coefficients Regularized} that
\[
	a_\delta(\bx) \;=\; \varepsilon + \varepsilon \chi_\delta(\bx)\bigl(E(\bx) - 1\bigr),
\]
where $E(\bx) = e^{(F(\bx) - F_0)/\varepsilon}$. Therefore
\begin{equation}
\label{eq:grad-aDelta}
\nabla a_\delta \;=\; \varepsilon (E - 1) \nabla \chi_\delta
   + \varepsilon \chi_\delta \nabla E.
\end{equation}
For the first term in~\eqref{eq:grad-aDelta}, recall from the proof of Lemma~\ref{LMMA:Uniform Coeff}
that 
$\bigl|\varepsilon (E - 1) \nabla \chi_\delta\bigr|
\;\leq\; M_2$ for some $M_2>0$ on $U_\delta = \{\bx : |\mathrm{dist}(\bx, \Gamma)| < \delta\}$, with the bound being zero outside $U_\delta$.

For the second term in~\eqref{eq:grad-aDelta}, $\chi_\delta$ is uniformly bounded by $1$
and $|\nabla E| = (1/\varepsilon)|E| |\nabla F|$ is uniformly bounded
on $K$ by Lemma~\ref{LMMA:Uniform Coeff} and ~\hyperlink{F-a}{\bf (F-a)}. Therefore, we have
$\bigl|\varepsilon \chi_\delta \nabla E\bigr|
\;\leq\; \|E\|_{L^\infty(K)} \|\nabla F\|_{L^\infty(K)}
\;=:\; M_3$. We, therefore, have
\[
|\nabla a_\delta(\bx)| \;\leq\; M_2 \mathbf{1}_{U_\delta}(\bx) + M_3,
\qquad \bx \in K\,.
\]
This leads to
\begin{multline}
\label{EQ:Bound T2}
\int_0^T\!\!\int_K \rho_\delta^2 |\nabla a_\delta|^2 \, \md\bx \, dt
\;\leq\; 2 M_2^2 \int_0^T\!\!\int_{K \cap U_\delta} \rho_\delta^2 \, \md\bx \, dt
   + 2 M_3^2 \int_0^T\!\!\int_K \rho_\delta^2 \, \md\bx \, dt
\\
\;\leq\; 2 \bigl(M_2^2 + M_3^2\bigr) T
   \|\rho_\delta\|_{L^\infty((0,T) \times K)}^2 |K|
\;\leq\; 2 \bigl(M_2^2 + M_3^2\bigr) T (C_{T,K}^{(1)})^2 |K|
\;=:\; C_{T,K}^{(2)},
\end{multline}
where in the second inequality we used the $L^\infty$ bound on
$\rho_\delta$ and the trivial bound $|U_\delta \cap K| \leq |K|$.

Inserting~\eqref{EQ:Bound T1} and~\eqref{EQ:Bound T2} into~\eqref{EQ:Grad-aRho-Split} and using $(u+v)^2 \leq 2u^2 + 2v^2$, we obtain~\eqref{EQ:H1-aRho} with
\[
\widetilde{C}_{T,K}
\;=\; 2 \Bigl[
   \frac{2 \Lambda_K^2 C_{T,K}^{(1)}}{a_{\min}} \mathrm{KL}(\rho_0\|\pi)
   + 2 \Lambda_K^2 T C_{T,K}^{(1)} \|\nabla \log \pi\|_{L^\infty(K)}^2
   + C_{T,K}^{(2)}
\Bigr].
\]
The proof is complete since all constants on the right are independent of $\delta$.
\end{proof}

The next step is to show the following time-derivative bound in $H^{-1}_{\mathrm{loc}}$.
\begin{lemma}
\label{LMMA:Time-Deri Bound}
Assume that $F$ satisfies assumption~\hyperlink{F-a}{\bf (F-a)} and $\rho_{0}$ satisfies ~\eqref{EQ:IC ASS}. Then the following holds.
\begin{itemize}
\item In \hyperlink{case1}{$\CaseI$}, $\{\partial_t\rho_\delta\}_{\delta>0}$ is bounded in
$L^2((0,T);H^{-1}(\Omega))$ uniformly in $\delta$.
\item In \hyperlink{case2}{$\CaseII$}, for every $R>0$,
$\{\partial_t\rho_\delta\}_{\delta>0}$ is bounded in
$L^2((0,T);H^{-1}(B_R))$ uniformly in $\delta$.
\end{itemize}
\end{lemma}
\begin{proof}
We use the entropy form~\eqref{EQ:FP Onsager Form}. Let $K=\Omega$ in \hyperlink{case1}{$\CaseI$} and $K=B_R$ in \hyperlink{case2}{$\CaseII$}. For any $\varphi\in H^1(\Omega)$ with $\bn\cdot\nabla \varphi_{|\partial\Omega}=0$ in \hyperlink{case1}{$\CaseI$} or $\varphi\in H^1_0(K)$ in \hyperlink{case2}{$\CaseII$}, we have
\[
\langle \partial_t\rho_\delta,\varphi\rangle
=
-\int_{K} a_\delta\rho_\delta\,\nabla\log\!\left(\frac{\rho_\delta}{\pi}\right)\cdot\nabla\varphi\,\md \bx.
\]
By the Cauchy--Schwarz and the assumption that $a_\delta\le \Lambda_K$ in~\eqref{EQ:Uniform Bound Upper Bound AB} and~\eqref{EQ:Uniform Bound Upper Bound AB Local}, we have,
\[
\int_{K} a_\delta\rho_\delta|\nabla\varphi|^2\md \bx \leq \Lambda_K \|\rho_\delta\|_{L^\infty(K)} \|\nabla \varphi\|^2_{L^2(K)} \leq \Lambda_K C_{T,K}
\|\nabla\varphi\|^2_{L^2(K)} 
\]
for some constant $C_{T,K}$, using the uniform $L^\infty$ bound on $\rho_\delta$ from parabolic regularity (established in the proof of Lemma~\ref{LMMA:H1-aRho}). As a result,
\begin{multline*}
|\langle \partial_t\rho_\delta,\varphi\rangle|
\le
\left(\int_{K} a_\delta\rho_\delta\left|\nabla\log\!\left(\frac{\rho_\delta}{\pi}\right)\right|^2\md \bx\right)^{1/2}
\left(\int_{K} a_\delta\rho_\delta|\nabla\varphi|^2\md \bx\right)^{1/2}\\
\le
\sqrt{\Lambda_K C_{T,K}}\left(\int_{K} a_\delta\rho_\delta\left|\nabla\log\!\left(\frac{\rho_\delta}{\pi}\right)\right|^2\md \bx\right)^{1/2}
\|\nabla\varphi\|_{L^2(K)},
\end{multline*} 
Taking the supremum over $\varphi$ with $\|\nabla\varphi\|_{L^2}=1$ then gives
\[
\|\partial_t\rho_\delta\|_{H^{-1}(K)}
\le
\sqrt{\Lambda_K C_{T,K}} \left(\int_{K} a_\delta\rho_\delta\left|\nabla\log\!\left(\frac{\rho_\delta}{\pi}\right)\right|^2\md \bx\right)^{1/2}.
\]
We can now integrate this in time over the interval $(0,T)$ and use~\eqref{EQ:Fisher Bound B by IC} to bound the right-hand side uniformly in $\delta$.
\end{proof}

\subsubsection{The sharp-interface limit}

We now show that $\rho_\delta(t)$ converges to $\rho(t)$, in an appropriate sense, for some $\rho(t)$ that solves~\eqref{EQ:FP HB} in the distributional sense.

\begin{definition}[Distributional solution to the Fokker--Planck equation]
\label{DEF:Weak Solution}
Let $\rho_0\in L^1(\Omega)$ with $\rho_0\geq 0$. A nonnegative function
$\rho\in L^\infty((0,T);L^1(\Omega))$ is a \emph{distributional solution}
of~\eqref{EQ:FP HB} (with~\eqref{EQ:FP BC} in \hyperlink{case1}{$\CaseI$}) with initial
datum $\rho_0$ if, for every test function $\varphi$ in the class
\[
\cT_I:=\big\{\varphi\in C^\infty([0,T]\times\overline{\Omega}):
\varphi(T,\cdot)=0,\ \partial_n\varphi=0\text{ on }\partial\Omega\big\}
\]
in \hyperlink{case1}{$\CaseI$}, respectively $\cT_{II}:=C^\infty_c([0,T)\times\bbR^d)$ in
\hyperlink{case2}{$\CaseII$},
\begin{equation}
\label{EQ:Weak Form}
\int_0^T\!\!\int_\Omega\rho\,\partial_t\varphi\,\md\bx\,dt
+\int_\Omega\rho_0\,\varphi(0,\cdot)\,\md\bx
+\int_0^T\!\!\int_\Omega b\,\rho\cdot\nabla\varphi\,\md\bx\,dt
+\int_0^T\!\!\int_\Omega a\,\rho\,\Delta\varphi\,\md\bx\,dt=0.
\end{equation}
\end{definition} 

\begin{remark}
The interface conditions are not imposed separately in Definition~\ref{DEF:Weak Solution}. They are encoded in the global weak formulation. If the limiting density is sufficiently regular on both sides of $\Gamma$, then splitting~\eqref{EQ:Weak Form} over $\Omega_-$ and $\Omega_+$ and integrating by parts gives the interface terms
\[
\int_\Gamma [(a\rho)]_\Gamma\,\partial_{\bn_\Gamma}\varphi\,dS
+
\int_\Gamma [\bn_\Gamma\cdot J(\rho)]_\Gamma\,\varphi\,dS .
\]
Since $\varphi|_\Gamma$ and $\partial_{\bn_\Gamma}\varphi|_\Gamma$ can be varied independently, the classical transmission conditions are
\[
    [a\rho]_\Gamma=0,
    \qquad
    [\bn_\Gamma\cdot J(\rho)]_\Gamma=0.
\]
In the present construction $a_-|_\Gamma=a_+|_\Gamma=\varepsilon$, so $[a\rho]_\Gamma=0$ is equivalent to equality of the traces of $\rho$. \end{remark}

It turns out that the limit of $\rho_\delta$ as $\delta\to 0$ is indeed a distributional solution to~\eqref{EQ:FP HB}.
\begin{theorem}\label{THM:Limit Equation}
Assume $F$ satisfies~\hyperlink{F-a}{\bf (F-a)}, the LSI~\eqref{EQ:LSI}, and, in {\rm \hyperlink{case2}{$\CaseII$}}, the exponential moment
condition~\eqref{EQ:Moment Cond}. Assume further that $\rho_0$ satisfies~\eqref{EQ:IC ASS}. Let $\rho_\delta$ be the classical
smooth solution of~\eqref{EQ:FP Regularized} with initial datum
$\rho_0$. Then there exist $\rho\in L^\infty((0,T);L^1(\Omega))$ and
a subsequence $\delta_k\to 0$ (not relabeled) such that:
\begin{enumerate}
\item[(i)] For every compact $K\Subset\Omega$ (with $K=\Omega$ in \hyperlink{case1}{$\CaseI$}),
\[
\rho_\delta\to\rho\text{ strongly in }L^1((0,T)\times K)
\text{ and a.e.\ on }(0,T)\times K.
\]
\item[(ii)] $\rho$ is a distributional solution of the hybrid
transmission problem~\eqref{EQ:FP HB} with initial datum $\rho_0$
in the sense of Definition~\ref{DEF:Weak Solution}.
\item[(iii)] $\rho(t)\rightharpoonup\rho_0$ weakly in $L^1(\Omega)$ as $t\downarrow 0$.
\end{enumerate}
\end{theorem}

\begin{proof} 

(i): Fix a compact set $K\Subset\Omega$ (with $K=\Omega$ in \hyperlink{case1}{$\CaseI$}). By Lemma~\ref{LMMA:Grad Bound}, $\{\sqrt{\rho_\delta}\}$ is bounded in $L^2((0,T);H^1(K))$. Together with the uniform local $L^\infty$ bound on $\rho_\delta$ from Lemma~\ref{LMMA:Linfty}, this implies
\[
\int_0^T\!\!\int_K |\nabla\rho_\delta|^2\,\md\bx\,dt
=
4\int_0^T\!\!\int_K \rho_\delta
|\nabla\sqrt{\rho_\delta}|^2\,\md\bx\,dt
\le C_{T,K},
\]
and, using mass conservation, also gives a uniform bound for $\rho_\delta$ in $L^2((0,T);H^1(K))$. By Lemma~\ref{LMMA:Time-Deri Bound},
$\{\partial_t\rho_\delta\}$ is bounded in $L^2((0,T);H^{-1}(K))$. The Aubin--Lions lemma, with $H^1(K)\hookrightarrow\hookrightarrow L^2(K)\hookrightarrow H^{-1}(K)$, therefore gives, along a subsequence, $\rho_\delta\to\rho$ strongly in $L^2((0,T);L^2(K))$. In particular, $\rho_\delta\to\rho$ strongly in $L^1((0,T)\times K)$ and, after passing to a further subsequence, almost everywhere on $(0,T)\times K$. A diagonal argument over compact sets $K_j\uparrow\Omega$ produces a single subsequence.

(ii): We fix an admissible test function $\varphi$ as in
Definition~\ref{DEF:Weak Solution}. Since $\rho_\delta$ is a
classical solution of~\eqref{EQ:FP Regularized}, multiplying by
$\varphi$ and integrating by parts (using either the no-flux condition
and $\partial_n\varphi=0$ on $\partial\Omega$ in \hyperlink{case1}{$\CaseI$} or the compact
support of $\varphi$ in \hyperlink{case2}{$\CaseII$}), yields the identity
\begin{equation}
\label{EQ:rho-delta Weak Form}
\int_0^T\!\!\int_\Omega\rho_\delta\,\partial_t\varphi
+\int_\Omega\rho_0\,\varphi(0,\cdot)
+\int_0^T\!\!\int_\Omega b_\delta\rho_\delta\cdot\nabla\varphi
+\int_0^T\!\!\int_\Omega a_\delta\rho_\delta\Delta\varphi=0.
\end{equation}
To pass to the limit, let $K\Subset\Omega$ contain
$\supp\varphi(t,\cdot)$ for all $t$ (with $K=\Omega$ in \hyperlink{case1}{$\CaseI$}). On
$(0,T)\times K$ we have $\rho_\delta\to\rho$ strongly in $L^1$ (by (i)), $a_\delta\to a$ and $b_\delta\to b$ a.e.\ with uniform $L^\infty(K)$ bounds (Lemma~\ref{LMMA:Uniform Coeff}), and $\partial_t\varphi$,
$\nabla\varphi$, $\Delta\varphi$ are bounded. Hence
\[
a_\delta\rho_\delta\to a\rho,\quad
b_\delta\rho_\delta\to b\rho
\quad\text{strongly in }L^1((0,T)\times K),
\]
and each integral in~\eqref{EQ:rho-delta Weak Form} converges to
the corresponding one in~\eqref{EQ:Weak Form}.

(iii): We test~\eqref{EQ:Weak Form} against
$\varphi(t,\bx)=\eta(t)\psi(\bx)$ with $\eta\in C^\infty([0,T])$
supported near $t=0$ with $\eta(0)=1$ and $\psi$ in the appropriate
spatial test class. The uniform integrability of $\{\rho(t)\}$ near $t=0$ inherited from~\eqref{EQ:KL Bound by IC} (via the de~la~Vall\'ee--Poussin criterion), together with the mass normalization, identifies $\rho_0$ as the weak-$L^1$ initial trace and yields weak $L^1$-continuity at $t=0$. 
\end{proof}

It is important to realize that the transmission conditions across $\Gamma$ in~\eqref{EQ:FP HB} are encoded in the weak formulation. Whenever the limiting density has enough regularity for traces to be defined, this weak formulation recovers the classical transmission conditions across $\Gamma$. Moreover, we can use Lemma~\ref{LMMA:H1-aRho} to deduce a corresponding $H^1$ regularity statement for the limit $\rho$, and consequently the continuity of $\rho$ across the interface $\Gamma$.

\begin{theorem}\label{THM:Continuity of Rho}
Let $\rho$ be the distributional solution of the hybrid
transmission problem~\eqref{EQ:FP HB} provided by
Theorem~\ref{THM:Limit Equation}. Then:
\begin{enumerate}
\item[\textup{(i)}] $a \rho \in L^2\bigl((0,T); H^1_{\mathrm{loc}}(\Omega)\bigr)$,
   with $a\rho \in L^2\bigl((0,T); H^1(\Omega)\bigr)$ in~\hyperlink{case1}{$\CaseI$}.
\item[\textup{(ii)}] For almost every $t \in (0,T)$, the trace of $\rho(t,\cdot)$ on $\Gamma$ from $\Omega_-$ and from $\Omega_+$ coincide as elements of $H^{1/2}(\Gamma)$, that is,
   \begin{equation}\label{EQ:Continuity of Rho}
   \bigl[\rho\bigr]_\Gamma \;:=\;
   \rho\big|_{\Gamma^-} - \rho\big|_{\Gamma^+} \;=\; 0
   \quad \text{in } H^{1/2}(\Gamma).
   \end{equation}
\end{enumerate}
\end{theorem}

\begin{proof}
First, by Lemma~\ref{LMMA:H1-aRho}, the family
$\{a_\delta \rho_\delta\}_{\delta > 0}$ is uniformly bounded in $L^2\bigl((0,T); H^1(K)\bigr)$ for every compact $K \Subset \Omega$ (with $K = \Omega$ in~\hyperlink{case1}{$\CaseI$}). By the Banach--Alaoglu theorem, along a subsequence (which, by uniqueness of the limit we have established, can be taken to be the same as in
Theorem~\ref{THM:Limit Equation}), there exists
$w \in L^2\bigl((0,T); H^1(K)\bigr)$ such that
\begin{equation*}
a_\delta \rho_\delta \;\rightharpoonup\; w
\quad \text{weakly in } L^2\bigl((0,T); H^1(K)\bigr).
\end{equation*}

We claim that $w = a \rho$ almost everywhere on $(0,T) \times K$. Indeed, by Theorem~\ref{THM:Limit Equation} (i), $\rho_\delta \to \rho$ strongly in $L^1((0,T) \times K)$, and by Lemma~\ref{LMMA:Uniform Coeff} (iv), $a_\delta \to a$ almost everywhere on $\Omega$. Combined with the uniform $L^\infty$ bound on
$a_\delta$, the dominated convergence theorem yields
\[
a_\delta \rho_\delta \;\longrightarrow\; a \rho
\quad \text{in } L^1((0,T) \times K)
\]
and hence in the sense of distributions on $(0,T) \times K$. The weak limit in $L^2_t H^1_{\bx}$ must agree with the distributional
limit, so $w = a \rho$ almost everywhere. A standard diagonal argument over an exhaustion $K_j \uparrow \Omega$ extends this to $a \rho \in L^2\bigl((0,T); H^1_{\mathrm{loc}}(\Omega)\bigr)$. In~\hyperlink{case1}{$\CaseI$}, $K = \Omega$ and we obtain the global statement $a\rho \in L^2((0,T); H^1(\Omega))$.

To prove (ii), we fix $t \in (0,T)$ such that
$a(t,\cdot) \rho(t,\cdot) \in H^1_{\mathrm{loc}}(\Omega)$. By Fubini's theorem and part~(i), this holds for almost every $t$. Choose an open neighborhood $\mathcal{V}$ of $\Gamma$ with
$\mathcal{V} \Subset \Omega$, such that $\Gamma$ separates $\mathcal{V}$ into two open sets
$\mathcal{V}_\pm := \mathcal{V} \cap \Omega_\pm$. Since $\Gamma \in C^1$, both $\mathcal{V}_\pm$ are Lipschitz domains (in fact $C^1$), and the trace operators $\gamma^\pm \colon H^1(\mathcal{V}_\pm) \to H^{1/2}(\Gamma)$ are well-defined and continuous. 

The restriction of $a\rho$ to $\mathcal{V}$ lies in $H^1(\mathcal{V})$, since $\mathcal{V} \Subset \Omega$ and $a\rho \in H^1_{\mathrm{loc}}(\Omega)$. We now use the following elementary fact: if $u \in H^1(\mathcal{V})$ and $\mathcal{V}$ is partitioned by a $C^1$ hypersurface $\Gamma$ into Lipschitz domains $\mathcal{V}_\pm$, then the traces $\gamma^+(u|_{\mathcal{V}_+})$ and $\gamma^-(u|_{\mathcal{V}_-})$ on $\Gamma$ coincide in $H^{1/2}(\Gamma)$. Indeed, otherwise the distributional gradient $\nabla u$, computed in $\mathcal{V}$, would contain a surface measure on $\Gamma$ of the form $[\gamma^+(u) - \gamma^-(u)]\nu_\Gamma \, d\mathcal{H}^{d-1}|_\Gamma$, where $\nu_\Gamma$ is the unit normal. This contradicts $\nabla u \in L^2(\mathcal{V})$ (see, e.g.,~\cite[Lemma~5.4]{EvGa-Book15} or~\cite[Section 2.4]{Ziemer-Book12}). Applying this fact to $u = a(t,\cdot) \rho(t,\cdot)$ yields that, for almost every $t\in(0, T)$, we have
\[
 (a\rho)\big|_{\Gamma^-} - (a\rho)\big|_{\Gamma^+} \;=\; 0
   \quad \text{in } H^{1/2}(\Gamma)\,.
\]
We now use the assumption that $a|_\Gamma$ does not depend on whether we approach $\Gamma$ from $\Omega_-$ or from $\Omega_+$ (i.e., the trace of $a$ on $\Gamma$ from either side is the constant $\varepsilon$) to conclude the proof.
\end{proof}
In fact, the continuity statement~\eqref{EQ:Continuity of Rho} holds for almost every $t \in (0,T)$. After invoking the weak $L^1$-continuity result of Theorem~\ref{THM:Weak Continuity} in the next subsection, one may upgrade this to: for \emph{every} $t \in [0,T]$, the trace identity~\eqref{EQ:Continuity of Rho} holds in $H^{1/2}(\Gamma)$. Note also that the conclusion~\eqref{EQ:Continuity of Rho} concerns the trace of $\rho$ on $\Gamma$, not the trace of its gradient. The normal derivative $\bn_\Gamma \cdot \nabla \rho$ is generally \emph{not} continuous across $\Gamma$: indeed, by the flux-matching condition $[\bn_\Gamma \cdot J(\rho)]_\Gamma = 0$ encoded in~\eqref{EQ:FP HB}, together with the discontinuity
of $b$ across $\Gamma$, the normal flux $\nabla(a\rho) - b\rho$ must adjust its normal-derivative contribution to compensate for the jump in $b$.

\subsection{Exponential rate for hybrid dynamics}
\label{SUBSEC:Convergence Rate}

We now use the uniform entropy dissipation estimates and the result in the previous subsection to show that the exponential $\KL$ rate does not deteriorate as $\delta\to 0$, that is, the hybrid dynamics converges to the Gibbs distribution.

We first prove the following weak $L^1$-continuity result for the limit $\rho$, the weak solution to~\eqref{EQ:FP HB}.
\begin{theorem}
\label{THM:Weak Continuity}
Let $\rho \in L^\infty((0,T);L^1(\Omega))$ be the distributional solution
of~\eqref{EQ:FP HB} obtained in
Theorem~\ref{THM:Limit Equation}. Then $\rho$ admits a representative (still
denoted $\rho$) such that the map
\[
t \;\longmapsto\; \rho(t)
\]
is weakly continuous from $[0,T]$ into $L^1(\Omega)$, in the following
sense: for every $\psi$ in the spatial test class
$C^\infty(\overline{\Omega})$ with $\bn\cdot\nabla \psi=0$ on $\partial\Omega$ in \hyperlink{case1}{$\CaseI$} (respectively $\psi\in C_c^\infty(\bbR^d)$ in \hyperlink{case2}{$\CaseII$}), the map
\begin{equation}\label{EQ:t-to-rho test func}
t \;\longmapsto\; \int_\Omega \rho(t,\bx)\,\psi(\bx)\,\md\bx
\end{equation}
is continuous on $[0,T]$.

Moreover, for every $t\in[0,T]$, and along the subsequence
$\{\delta_k\}$ of Theorem~\ref{THM:Limit Equation}, we have that,
\begin{equation}
\label{EQ:Pointwise Weak Conv}
\rho_{\delta_k}(t) \;\rightharpoonup\; \rho(t)
\qquad \text{weakly in } L^1(\Omega).
\end{equation}
\end{theorem}
\begin{proof}
Fix $\psi$ in the spatial test class. For each fixed $\delta>0$ the map $t\mapsto\int_\Omega\rho_\delta(t)\psi\,\md\bx$ is $C^1$ on $[0,T]$: indeed $\rho_\delta\in C^\infty$ jointly in $(t,\bx)$ by Theorem~\ref{THM:FP Regularized Solution Theory}, so $\partial_t\rho_\delta$ is continuous and differentiation under the integral sign is justified by dominated convergence ($\Omega$ being bounded in \hyperlink{case1}{$\CaseI$}, and $\psi$ compactly supported in \hyperlink{case2}{$\CaseII$}). Differentiating and using~\eqref{EQ:FP Regularized} with integration by parts gives
\begin{equation}\label{EQ:Time-Derivative-G}
\frac{d}{dt}\int_\Omega\rho_\delta(t)\psi\,\md\bx
=\int_\Omega\bigl(b_\delta\rho_\delta\cdot\nabla\psi+a_\delta\rho_\delta\,\Delta\psi\bigr)\,\md\bx,
\end{equation}
the boundary terms vanishing by~\eqref{EQ:FP Regularized BC} and $\bn_{|\partial\Omega}\cdot\nabla\psi=0$ in \hyperlink{case1}{$\CaseI$}, or by compact support of $\psi$ in \hyperlink{case2}{$\CaseII$}. By Lemma~\ref{LMMA:Uniform Coeff} and mass conservation~\eqref{EQ:Mass Conservation}, the right-hand side of~\eqref{EQ:Time-Derivative-G} is bounded by a constant $C_\psi$ depending only on $\|\nabla\psi\|_{L^\infty}$, $\|\Delta\psi\|_{L^\infty}$, and the $L^\infty$ bounds on $(a_\delta,b_\delta)$ on $\mathrm{supp}\,\psi$, uniformly in $\delta$ and $t$. Hence the family $\{t\mapsto\int_\Omega\rho_\delta(t)\psi\,\md\bx\}_{\delta>0}$ is uniformly Lipschitz on $[0,T]$.

By Theorem~\ref{THM:Limit Equation} (i) and Fubini's theorem, $\int_\Omega\rho_{\delta_k}(t)\psi\,\md\bx\to\int_\Omega\rho(t)\psi\,\md\bx=:G_\psi(t)$ for a.e.\ $t\in(0,T)$. The uniform Lipschitz bound then forces $G_\psi$ to extend to a Lipschitz function on $[0,T]$ and the convergence to hold for \emph{every} $t\in[0,T]$. Moreover, by~\eqref{EQ:KL Bound by IC} and the upper bound on $\log\pi$ from~\hyperlink{F-a}{\bf (F-a)}, $\sup_{\delta>0}\int_\Omega\rho_\delta(t)\log\rho_\delta(t)\,\md\bx<\infty$ for each $t$, so $\{\rho_\delta(t)\}_\delta$ is uniformly integrable (de~la~Vall\'ee--Poussin) and, in \hyperlink{case2}{$\CaseII$}, tight by the moment bound of Lemma~\ref{LMMA:Moment}. The Dunford-Pettis theorem then gives weak-$L^1$ relative compactness. Since every weak-$L^1$ limit of $\{\rho_{\delta_k}(t)\}$ is pinned down on a countable dense subset of the test class by the values $G_\psi(t)$, it is unique. Hence, the full sequence satisfies $\rho_{\delta_k}(t)\rightharpoonup\rho(t)$ in $L^1(\Omega)$ for every $t\in[0,T]$, and (after redefining $\rho$ on the null set where needed) $t\mapsto\rho(t)$ is the asserted weakly continuous representative, with $\int_\Omega\rho(t)\psi\,\md\bx=G_\psi(t)$ continuous for every $\psi$ in the test class. This proves both~\eqref{EQ:t-to-rho test func} and~\eqref{EQ:Pointwise Weak Conv}. 
\end{proof}

We are finally ready to show that the hybrid dynamics converges to the equilibrium exponentially fast.
\begin{theorem}
\label{THM:Exponential Rate Limit}
Under the hypotheses of Theorem~\ref{THM:Limit Equation}, the distributional
solution $\rho$ of~\eqref{EQ:FP HB} (chosen as the weakly continuous representative provided by
Theorem~\ref{THM:Weak Continuity}) satisfies, for every $t\geq 0$,
\begin{equation}\label{EQ:Exp Rate HD}
\KL\bigl(\rho(t)\,\|\,\pi\bigr)
\;\leq\;e^{-2 a_{\min}\lambda_{\mathrm{LSI}}\,t}\,\KL(\rho_0\,\|\,\pi).
\end{equation}
\end{theorem}
\begin{proof}
Fix $t\geq 0$. By Theorem~\ref{THM:Weak Continuity}, $\rho_{\delta_k}(t)\rightharpoonup\rho(t)$ weakly in $L^1(\Omega)$ along the subsequence of Theorem~\ref{THM:Limit Equation}. Lower semicontinuity of $\KL$ under weak $L^1$-convergence gives
\[
\KL\bigl(\rho(t)\,\|\,\pi\bigr)
\;\leq\;\liminf_{k\to\infty}\KL\bigl(\rho_{\delta_k}(t)\,\|\,\pi\bigr).
\]
The entropy decay estimate~\eqref{EQ:KL Decay} applied to each smooth classical solution $\rho_{\delta_k}$ yields
\[
\KL\bigl(\rho_{\delta_k}(t)\,\|\,\pi\bigr)
\;\leq\;e^{-2a_{\min}\lambda_{\mathrm{LSI}}\,t}\,\KL(\rho_0\,\|\,\pi),
\]
where the right-hand side is independent of $k$ because the initial datum is independent of $\delta$. Taking $\liminf$ gives~\eqref{EQ:Exp Rate HD}.
\end{proof}
\begin{remark}
    It is clear that Theorem~\ref{THM:Exponential Rate Limit} holds for a.e. $t\ge 0$ without the weak continuity result in Theorem~\ref{THM:Weak Continuity}. The strong $L^1((0,T)\times K)$ convergence gives pointwise $L_{\rm loc}^1$ convergence for a.e. $t$, which is stronger than weak convergence. Hence, $\KL$ lower semicontinuity applies. The weak continuity result of Theorem~\ref{THM:Weak Continuity} upgrades ``a.e. $t$" to ``every $t$".
\end{remark}

The rate from Theorem~\ref{THM:Exponential Rate Limit} is $2a_{\min}\lambda_{\text{LSI}}$ where $a_{\min} = \eps \min\{1, e^{(F_{\min} - F_0)/\eps}\}$. If the interior well is deeper than the switching level (that is, $F_{\min}<F_0$, which is the interesting case, as otherwise the hybridization region does not contain the wells), $a_{\min}$ is exponentially small in $1/\eps$, and the rate from Theorem~\ref{THM:Exponential Rate Limit} is worse than the analogous bound for pure Langevin. Thus, Theorem~\ref{THM:Exponential Rate Limit} should be understood as a qualitative ``rate does not deteriorate under regularization" result rather than a quantitative improvement over pure Langevin. The genuine quantitative advantage of the hybrid construction is captured at the level of mean exit times and is established in Section~\ref{SEC:Exit} next.

\section{Mean exit time comparison}
\label{SEC:Exit}

The analysis in the previous sections shows that the hybrid dynamics preserves the Gibbs distribution and converges to equilibrium at an exponential rate comparable to that of the regularized system. The entropy-based convergence results, however, do not fully capture the mechanism responsible for slow mixing in multimodal landscapes, that is, the presence of metastable states separated by energy barriers.

We now complement the entropy-based analysis by studying the metastability properties of the hybrid dynamics. In particular, we investigate how the hybrid construction modifies the transition mechanism between metastable wells and whether this can potentially lead to an improvement in transition times. We will focus on the limit object whose existence was established in Theorem~\ref{THM:Limit Equation}. Therefore, we will use the explicit discontinuous coefficients without further justification.
 
We now introduce the second partitioning convention. In the preceding sections, the switching interface was the full level set
\[
    \Gamma_{F_0}:=\{\bx\in\Omega:F(\bx)=F_0\},
\]
which separates $\Omega$ into two regions. In the metastability setting considered below, the potential is radially symmetric and may be nonconvex, so $\Gamma_{F_0}$ may have several connected components. We therefore take the switching interface to be only the outermost connected component, denoted by $\Gamma_{F_0}^{\mathrm{out}}$, namely the component adjacent to the exterior region, and set $\Gamma:=\Gamma_{F_0}^{\mathrm{out}}$.
The regions $\Omega_-$ and $\Omega_+$ are then defined as the interior and exterior regions separated by this component: $\Omega_-$ is the region enclosed by $\Gamma$, while $\Omega_+$ is the exterior region. The hybrid coefficients are defined by the same sharp-interface construction as before, but with this outermost component serving as the interface. Since $F$ is radially symmetric in this setting, the partition can equivalently be described in terms of the radial variable $r=\|\bx\|$.

\subsection{Radially symmetric setting}
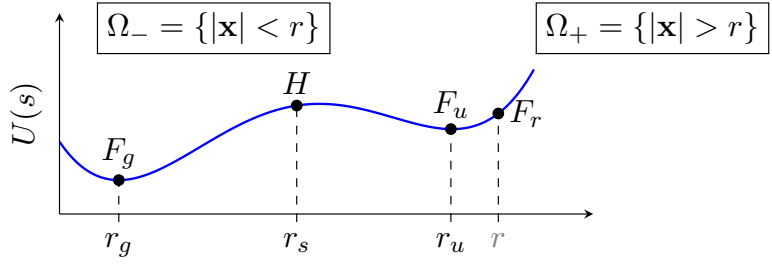
\begin{figure}
\centering
\begin{tikzpicture}[scale=1.1, transform shape]
\begin{axis}[
    width=8cm,
    height=4cm,
    xmin=0, xmax=4.5,
    ymin=0.2, ymax=5.2,
    axis lines=left,
    ylabel={$U(s)$},
    xtick=\empty,
    ytick=\empty,
    clip=false,
    domain=0:4,
    samples=200
]

\def\f{0.3*(x-2)^4 - 1.2*(x-2)^2 + 0.45*x + 2}
\addplot[thick, blue] {\f};

\def\rhog{0.5}
\def\rhos{2.0}
\def\rhou{3.3}

\pgfmathsetmacro{\Ug}{0.3*(\rhog-2)^4 - 1.2*(\rhog-2)^2 + 0.45*\rhog + 2}
\pgfmathsetmacro{\Us}{0.3*(\rhos-2)^4 - 1.2*(\rhos-2)^2 + 0.45*\rhos + 2}
\pgfmathsetmacro{\Uu}{0.3*(\rhou-2)^4 - 1.2*(\rhou-2)^2 + 0.45*\rhou + 2}

\def\rR{3.7}
\pgfmathsetmacro{\Ur}{0.3*(\rR-2)^4 - 1.2*(\rR-2)^2 + 0.45*\rR + 2}
\node[right] at (axis cs:\rR,\Ur) {$F_r$};
\node[below, gray] at (axis cs:\rR,0) {$r$};

\addplot[only marks, mark=*, mark size=1.8pt] coordinates {
    (\rhog,\Ug)
    (\rhos,\Us)
    (\rhou,\Uu)
    (\rR,\Ur)
};

\addplot[dashed] coordinates {(\rhog,0) (\rhog,\Ug)};
\addplot[dashed] coordinates {(\rhos,0) (\rhos,\Us)};
\addplot[dashed] coordinates {(\rhou,0) (\rhou,\Uu)};
\addplot[dashed] coordinates {(\rR,0) (\rR,\Ur)};

\node[below] at (axis cs:\rhog,0) {$r_g$};
\node[below] at (axis cs:\rhos,0) {$r_s$};
\node[below] at (axis cs:\rhou,0) {$r_u$};

\node[above] at (axis cs:\rhog,\Ug) {$F_g$};
\node[above] at (axis cs:\rhos,\Us) {$H$};
\node[above] at (axis cs:\rhou,\Uu) {$F_u$};

\node[above] at (axis cs:1.3,3.9) {$\boxed{\Omega_-=\{|\bx|<r\}}$};
\node[above] at (axis cs:5,3.9) {$\boxed{\Omega_+=\{|\bx|>r\}}$};

\end{axis}
\end{tikzpicture}

\caption{A schematic radial potential with global minimum $r_g$, saddle $r_s$, local minimum $r_u$, and switching radius $r$.}
\label{fig:doublewell}
\end{figure}

We will compare the hybrid dynamics with the standard overdamped Langevin dynamics in a radially symmetric landscape. We consider the case when the state space $\Omega=\bbR^d$. 

Let $r>0$ be given, and $U(s)\in \cC^2([0,\infty))$ be a potential that has a double-well structure on $[0,r)$ and is strictly convex on $[r,\infty)$; see Figure~\ref{fig:doublewell}. We construct the potential $F$ from $U$ as
\[ 
    F(\bx)=U(|\bx|)\,.
\]
We assume that $U'(0)=0$ and that $U$ satisfies the usual even radial compatibility conditions at $s=0$, ensuring that $F(\bx)=U(|\bx|)$ is $\cC^2$ on $\mathbb R^d$.
The switching surface is the sphere
\[
\Gamma_r:=\{\bx\in\mathbb R^d:|\bx|=r\},
\] which separates the inner and outer regions
\[
\Omega_-:=\{\bx\in\mathbb R^d:|\bx|<r\},
\qquad
\Omega_+:=\{\bx\in\mathbb R^d:|\bx|>r\}.
\] 
The hybrid dynamics~\eqref{EQ:SDE HB} coincides with overdamped Langevin dynamics in $\Omega_+$ and replaces the inner dynamics by a drift-free diffusion with appropriately chosen diffusion coefficient:
\[
a(\bx) =
\begin{cases}
a_-:=\eps e^{(F(\bx)-F_r)/\eps}, & \bx \in \Omega_-,\\
a_+:=\eps, & \bx \in \Omega_+,
\end{cases}
\qquad
b(\bx) =
\begin{cases}
b_-:=0, & \bx \in \Omega_-,\\
b_+:=-\nabla F(\bx), & \bx \in \Omega_+,
\end{cases}
\]
where
\[
F_r:=U(r),
\]
so that $F(\bx)=F_r$ for every $\bx\in\Gamma_r$.

With the potential $U$ selected here, the overdamped Langevin dynamics is left unchanged outside the sphere $\Gamma_r$, where the potential is confining and already drives the process back toward the interior. 
The main question is whether the switching of the dynamics to the adaptive diffusion in $\Omega_-$ improves the mechanism responsible for transitions and, if so, what it does to the overall convergence to equilibrium.

\subsection{The conductance density \texorpdfstring{$a\pi$}{a pi}}

For a reversible diffusion with generator~\eqref{EQ:Generator}
and invariant density $\mu(\md\bx)=\pi(\bx)\,\md\bx$, the associated Dirichlet form is
\[
\mathcal E(f,f)
=
\int_{\mathbb R^d}a(\bx)|\nabla f(\bx)|^2\,d\mu(\bx)
=
\int_{\mathbb R^d}a(\bx)\pi(\bx)|\nabla f(\bx)|^2\,\md\bx.
\] Thus the relevant coefficient in all variational problems is the product
\[ w(\bx):=a(\bx)\pi(\bx),
\] which we view as the \emph{conductance density}. Regions where $w$ is large make variations of the committor energetically expensive, whereas regions where $w$ is small allow the committor to change at low Dirichlet cost. Thus, small conductance regions form bottlenecks for transitions. This point is worth emphasizing: from the perspective of capacities, committors, and relaxation estimates, the process does not see $a$ and $\pi$ separately. What enters the variational formulas is their product $a\pi$. Therefore, the right quantity to compare between the two dynamics is not the diffusion coefficient $a$ alone but the conductance density $a\pi$. 

For the hybrid dynamics, the matching condition gives
\[
a(\bx)\pi(\bx)
=
\frac{\eps}{Z}e^{-F_r/\eps}
\qquad\text{for }\bx\in\Omega_-.
\] In particular, the conductance density is constant in the entire inner region. This simple identity is the main structural reason that the hybrid dynamics behaves differently from overdamped Langevin. The dynamics still preserves the same invariant law, but the geometry relevant to transitions is no longer governed by the wells and saddle of $F$ in the same way as in overdamped Langevin dynamics.

\subsection{Mean exit time comparison}

We next compare the transition from the upper well to the global minimum. The comparison is intended to isolate the main advantage of the hybrid construction. In a double-well landscape, the difficult event is the passage from the upper local minimum to the deeper well. For overdamped Langevin, this passage is controlled by the saddle height. For the hybrid dynamics, the expectation is that the saddle no longer appears in the leading exponential barrier.

We consider a radially symmetric potential with both local minima and global minima. Assume that $U\in \cC^2([0,\infty))$,  $s^{d-1}e^{-U(s)/\eps}\in L^1((0,\infty))$ for all sufficiently small $\eps>0$, and that there exist
radii
\[
0<r_g<r_s<r_u<r
\]
such that:
\begin{enumerate}[label=(\roman*)]
\item $U'(r_g)=U'(r_s)=U'(r_u)=0$ and these are the only critical points of $U$
in $(0,r]$;
\item the critical points are nondegenerate, namely, $U''(r_g)>0$, $U''(r_u)>0$ and $U''(r_s)<0$;
\item $U'(s)>0$ on $(r_g,r_s)\cup(r_u,\infty)$ and $U'(s)<0$ on
$(0,r_g)\cup(r_s,r_u)$;
\item  $U(s)\ge c_0s^\alpha-C_0$ for all sufficiently large $s$, for some $c_0,\alpha>0$ and $C_0\in\mathbb R$.
\end{enumerate}
We set
\[
F_g:=U(r_g),\qquad F_u:=U(r_u),\qquad H:=U(r_s),\qquad F_r:=U(r).
\]
and assume $F_g < F_u < H$. See Figure~\ref{fig:doublewell} for an illustration of the function $U$.

The transition of interest starts from the upper well shell
\[
\{\bx\in \mathbb{R}^d: |\bx|=r_u\}
\]
and ends on the lower well shell
\[
\{\bx\in \mathbb{R}^d: |\bx|=r_g\}.
\]

\begin{theorem}[Mean transition time comparison]\label{thm:time}
Let $\bx_u\in\bbR^d$ satisfy $|\bx_u|=r_u$, and define
\[
    \tau_g:=\inf\{t\ge0:\ |\bX_t|=r_g\}.
\]
As $\eps\to0$, the mean transition times satisfy
\[
\mathbb E_{\bx_u}\tau_g^{\mathrm{OL}}
\sim
\frac{2\pi}{\sqrt{U''(r_u)\,|U''(r_s)|}}
\left(\frac{r_u}{r_s}\right)^{d-1}
\exp\!\left(\frac{H-F_u}{\eps}\right),
\]
whereas
\[
\mathbb E_{\bx_u}\tau_g^{\mathrm{HD}}
\sim
\frac{1}{U''(r_g)}
\exp\!\left(\frac{F_r-F_g}{\eps}\right).
\]
Here, $\mathbb E_{\bx_u}\tau_g^{\mathrm{OL}}$ and
$\mathbb E_{\bx_u}\tau_g^{\mathrm{HD}}$ denote the mean transition times for the
overdamped Langevin and hybrid dynamics, respectively, started from
$\bx_u\in\{\bx\in\mathbb R^d:|\bx|=r_u\}$.
\end{theorem}
\begin{proof}[Proof sketch]
By radial symmetry, the mean hitting times depend only on the radial variable
$s=|\bx|$. We write
\[
    T_\eps^{\mathrm{OL}}(s):=\mathbb E_s\tau_g^{\mathrm{OL}},
    \qquad
    T_\eps^{\mathrm{HD}}(s):=\mathbb E_s\tau_g^{\mathrm{HD}},
\]
and set
\[
    \rho_\eps(s):=s^{d-1}e^{-U(s)/\eps}.
\]
For a radial test function $f(\bx)=\phi(|\bx|)$, the radial overdamped
Langevin generator is
\[
L_{\rm rad}^{\mathrm{OL}}\phi(s)
=
\eps\phi''(s)
+
\left(\frac{(d-1)\eps}{s}-U'(s)\right)\phi'(s)
=
\frac{1}{\rho_\eps(s)}
\frac{d}{ds}\left(A_\eps^{\mathrm{OL}}(s)\phi'(s)\right),
\]
where
\[
    A_\eps^{\mathrm{OL}}(s)=\eps\rho_\eps(s).
\]
For the hybrid dynamics,
\[
L_{\rm rad}^{\mathrm{HD}}\phi(s)
=
\frac{1}{\rho_\eps(s)}
\frac{d}{ds}\left(A_\eps^{\mathrm{HD}}(s)\phi'(s)\right), \quad\text{where}\quad   A_\eps^{\mathrm{HD}}(s)
=
\begin{cases}
\eps e^{-F_r/\eps}s^{d-1}, & s<r,\\[0.5ex]
\eps \rho_\eps(s), & s\ge r.
\end{cases}
\]

The functions $T_\eps^{\mathrm{OL}}$ and $T_\eps^{\mathrm{HD}}$ solve
\[
    -1=L_{\rm rad}T_\eps(s),\qquad s>r_g,
    \qquad
    T_\eps(r_g)=0,
\]
with the natural zero-flux condition $A_\eps(s)T_\eps'(s)\to0$ as $s\to\infty$.  For the hybrid dynamics, this Poisson problem is understood in the transmission sense at $s=r$: $T_\eps$ is continuous and $A_\eps T_\eps'$ is continuous across $r$.

After solving both equations, we obtain
\[
T_\eps^{\mathrm{OL}}(r_u)
=
\frac1\eps
\int_{r_g}^{r_u}
\frac{e^{U(y)/\eps}}{y^{d-1}}
\left(\int_y^\infty z^{d-1}e^{-U(z)/\eps}\,dz\right)dy
\]
and
\[
T_\eps^{\mathrm{HD}}(r_u)
=
\frac{e^{F_r/\eps}}{\eps}
\int_{r_g}^{r_u}
y^{-(d-1)}
\left(\int_y^\infty z^{d-1}e^{-U(z)/\eps}\,dz\right)dy .
\]
Note that since $r_u<r$, the Green-function formula for $T_\eps^{\mathrm{HD}}(r_u)$ only evaluates $A_\eps^{\mathrm{HD}}(y)$ for $y\in[r_g,r_u]$, where the first branch of the conductivity $A_\eps^{\mathrm{HD}}$ applies. The exterior region still enters through the tail integral $\int_y^\infty \rho_\eps(z)\,dz$.

We first consider the overdamped Langevin formula. Let
\[
    I_\eps(y):=\int_y^\infty z^{d-1}e^{-U(z)/\eps}\,dz .
\]
Let $r_*\in(r_g,r_s)$ be the unique point satisfying $U(r_*)=F_u$. On compact subintervals of $(r_*,r_u)$, the minimum of $U$ on $[y,\infty)$ is attained at the upper well $r_u$, and Laplace's method gives, uniformly away from the endpoints,
\[
    I_\eps(y)
    \sim
    r_u^{d-1}
    \sqrt{\frac{2\pi\eps}{U''(r_u)}}
    e^{-F_u/\eps}.
\]
Substituting this into the outer integral, the dominant contribution comes
from the nondegenerate saddle $r_s$, where $U(r_s)=H$ and
$U''(r_s)<0$. A second application of Laplace's method gives, over a neighborhood $[r_s-\sigma, r_s+\sigma]$ of the saddle,
\[
\int_{r_s-\sigma}^{r_s+\sigma}
\frac{e^{U(y)/\eps}}{y^{d-1}}\,dy
\sim
r_s^{-(d-1)}
\sqrt{\frac{2\pi\eps}{|U''(r_s)|}}
e^{H/\eps}.
\]
All other parts of the $y$-integral are exponentially smaller, or at most
polynomial in $\eps^{-1}$, by one-sided Laplace estimates near $r_g$ and by
the strict energy gap away from the saddle. Therefore
\[
T_\eps^{\mathrm{OL}}(r_u)
\sim
\frac{2\pi}{\sqrt{U''(r_u)|U''(r_s)|}}
\left(\frac{r_u}{r_s}\right)^{d-1}
\exp\!\left(\frac{H-F_u}{\eps}\right).
\]

For the hybrid dynamics, Tonelli's theorem rewrites the Green formula as
\[
T_\eps^{\mathrm{HD}}(r_u)
=
\frac{e^{F_r/\eps}}{\eps}
\int_{r_g}^{\infty}
z^{d-1}e^{-U(z)/\eps}
\left(\int_{r_g}^{z\wedge r_u}y^{-(d-1)}\,dy\right)dz .
\]
Define
\[
    J(z):=\int_{r_g}^{z\wedge r_u}y^{-(d-1)}\,dy .
\]
As $z\downarrow r_g$,
\[
    J(z)=r_g^{-(d-1)}(z-r_g)+\cO((z-r_g)^2),
    \qquad
    z^{d-1}=r_g^{d-1}+\cO(z-r_g),
\]
and hence
\[
    z^{d-1}J(z)=(z-r_g)+\cO((z-r_g)^2).
\]
Since $r_g$ is a nondegenerate minimum,
\[
    U(z)
    =
    F_g+\frac12U''(r_g)(z-r_g)^2+o((z-r_g)^2).
\]
Thus the leading contribution to the hybrid mean hitting time is
\[
\frac{e^{F_r/\eps}}{\eps}e^{-F_g/\eps}
\int_0^\infty
\eta\,
e^{-U''(r_g)\eta^2/(2\eps)}\,d\eta
=
\frac{1}{U''(r_g)}
\exp\!\left(\frac{F_r-F_g}{\eps}\right).
\]
The complement of any fixed neighborhood of $r_g$ is exponentially smaller,
because $r_g$ is the unique global minimum on $[r_g,\infty)$ and the
confinement assumption controls the polynomial factor. Hence
\[
T_\eps^{\mathrm{HD}}(r_u)
\sim
\frac{1}{U''(r_g)}
\exp\!\left(\frac{F_r-F_g}{\eps}\right).
\]
This proves the two claimed asymptotic formulas.
\end{proof}
 
\begin{corollary}
Under the assumptions of Theorem~\ref{thm:time},
\[
\lim_{\eps\to0}
\eps\log\!\left(
\frac{\mathbb E_{\bx_u}\tau_g^{\mathrm{HD}}}
{\mathbb E_{\bx_u}\tau_g^{\mathrm{OL}}}
\right)
=
(F_r-F_g)-(H-F_u)
=
F_r+F_u-F_g-H.
\]
In particular, if
\begin{equation}\label{eq:key_relation}
F_r-F_g<H-F_u,
\end{equation}
then there exists $\eta>0$ such that
\[
\frac{\mathbb E_{\bx_u}\tau_g^{\mathrm{HD}}}
{\mathbb E_{\bx_u}\tau_g^{\mathrm{OL}}}
\le
\exp\!\left(-\frac{\eta}{\eps}\right)
\]
for all sufficiently small $\eps>0$.
\end{corollary}

\begin{remark}
For the overdamped Langevin dynamics, the metastable bottleneck is still the saddle at level $H$, so the relevant barrier is $H-F_u$. For the hybrid dynamics, once the drift has been removed in $\Omega_-$, the dominant cost is no longer the saddle-crossing cost but the cost of diffusing all the way to the lower minimum shell $\{|\bx|=r_g\}$, which produces the barrier $F_r-F_g$. Thus the hybrid dynamics is exponentially faster when~\eqref{eq:key_relation} holds.
\end{remark}

\begin{remark}
If the lower minimum is degenerate in the sense $U''(r_g)=0$, then one may replace the nondegeneracy assumption by the following: there exist an integer $m\ge 2$ and a constant $\kappa_g>0$ such that
\[
U(r_g+\eta)=F_g+\kappa_g \eta^{2m}+o(\eta^{2m})
\qquad (\eta\downarrow 0).
\]
Then the overdamped asymptotic is unchanged, as it is determined by the saddle at $r_s$.

For the hybrid dynamics, after the change of variables $\eta=z-r_g$,
\[
T_\eps^{\mathrm{HD}}(r_u)
\sim
\frac{e^{(F_r-F_g)/\eps}}{\eps}
\int_0^h \eta\,e^{-\kappa_g \eta^{2m}/\eps}\,d\eta
\]
for any fixed $h>0$ small enough, since the contribution of $[r_g+h,\infty)$ is exponentially smaller. With the substitution $t=\kappa_g \eta^{2m}/\eps$, one finds
\[
\int_0^\infty \eta\,e^{-\kappa_g \eta^{2m}/\eps}\,d\eta
=
\frac{1}{2m}\left(\frac{\eps}{\kappa_g}\right)^{1/m}
\int_0^\infty t^{1/m-1}e^{-t}\,dt
=
\frac{\Gamma(1/m)}{2m\,\kappa_g^{1/m}}\eps^{1/m}.
\]
Hence
\[
\mathbb E_{\bx_u}\tau_g^{\mathrm{HD}}
=
T_\eps^{\mathrm{HD}}(r_u)
\sim
\frac{\Gamma(1/m)}{2m\,\kappa_g^{1/m}}
\eps^{1/m-1}
\exp\!\left(\frac{F_r-F_g}{\eps}\right).
\]

Thus the exponent for the hybrid dynamics remains $F_r-F_g$, but the prefactor changes from a constant to a power of $\eps$. When  $m=1$, one recovers the case in Theorem~\ref{thm:time}.
\end{remark}

\section{Numerical simulations}
\label{SEC:Num}

We now present numerical simulations illustrating two features of the proposed hybrid dynamics. First, we verify empirically that the hybrid and regularized hybrid dynamics sample the desired Gibbs distribution. Second, we illustrate the accelerated redistribution of probability mass across metastable regions predicted by the mean-exit-time analysis of Section~\ref{SEC:Exit}.

We are interested in comparing results from the overdamped Langevin dynamics (OL), the hybrid dynamics~\eqref{EQ:SDE HB} (HD), and the regularized hybrid dynamics with regularization parameter $\delta$~\eqref{EQ:Coefficients Regularized} ($\delta$-HD). We discretize the continuous systems using the standard Euler-Maruyama scheme. In particular, with the time stepsize $\eta$, the overdamped Langevin becomes
\[
X_{n+1}=X_n-\eta \nabla F(X_n)+\sqrt{2\eta\eps}\,\xi_n,
\qquad \xi_n\sim \cN(\bzero,\bI).
\]
For the $\delta$-regularized hybrid dynamics, we have
\[
    X_{n+1}=X_n + \eta b_\delta(X_n)
+\sqrt{2\eta a_\delta(X_n)}\,\xi_n\,. 
\]
The limiting hybrid dynamics gives the piecewise Euler-Maruyama update
\[
X_{n+1}=
\begin{cases}
X_n+\sqrt{2\eta\,\eps\exp\left((F(X_n)-F_0)/\eps\right)}\,\xi_n,
& X_n \in\Omega_-,\\[1ex]
X_n-\eta \nabla F(X_n)+\sqrt{2\eta\eps}\,\xi_n,
& X_n \in\Omega_+,
\end{cases}
\]
where $F_0$ is the chosen value that determines the interface $\Gamma$.

\subsection{1D sampling examples}
\label{subsec:regularized-hd-1d}

We first show some one-dimensional simulations.
\begin{figure}[!htb]
    \centering
    \includegraphics[width=0.62\linewidth]{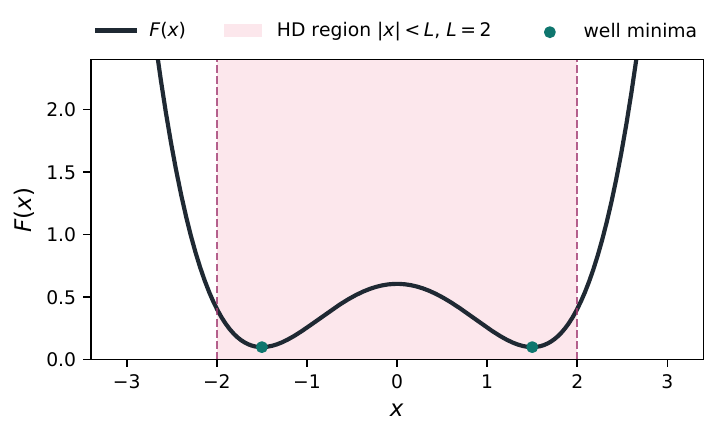}
    \caption{Double-well potential used in Numerical Example I. The shaded interval indicates the HD switching region $\Omega_-:\{x\in\bbR \mid |x|<L\}$, with
    $L=2$.  For the regularized HD methods, $\chi_\delta=1$ on
    $|x|\le L-\delta$ and decays smoothly to zero across the collar
    $L-\delta<|x|<L$.}
    \label{fig:regularized-hd-1d-potential}
\end{figure}

\paragraph{Example I (localized initialization).} In the first numerical example, we consider the
one-dimensional double-well potential (see Figure~\ref{fig:regularized-hd-1d-potential}):
\[
F(x)=\frac{1}{10}\bigl(x^2-c^2\bigr)^2+0.1,
\qquad c=1.5.
\]
The corresponding Gibbs equilibrium, for a given temperature parameter $\eps$, is
\[
\pi_G(x)=Z^{-1}\exp\{-F(x)/\eps\}\,. 
\]
We take a small $\eps=0.08$ to make the sampling challenging. The interior region for the derivative-free dynamics is taken to be $\Omega_-:=\{x\in\bbR \mid |x|<2\}$. The switching interface consists of the two outer points $x=\pm L$ with $L=2$. These points lie on the level $F_0=F(\pm L)=0.40625$. Thus this example uses the outermost-component convention: the hybridized region is the interval enclosed by the outermost level-set points.
  
For the coefficient regularization scheme~\eqref{EQ:Coefficients Regularized}, we take
\[
    S(t)=10t^3-15t^4+6t^5,\qquad 0\le t\le 1,
\]
and define the $C^2$ cutoff
\[
\chi_\delta(x)=
\begin{cases}
1, & |x|\le L-\delta,\\
1-S\!\left(\dfrac{|x|-(L-\delta)}{\delta}\right),
& L-\delta<|x|<L,\\
0, & |x|\ge L.
\end{cases}
\]
We test $\delta\in\{0.02,0.05,0.10,0.20\}$. The limiting HD method corresponds to the discontinuous-interface limit $\delta=0$.

The simulations are ``cold-started": they are initialized from a very concentrated distribution $\cN(3,0.01^2)$. All methods use the same time step $\eta=10^{-3}$, are initialized from $N(3,0.01^2)$, and are evolved up to $T=2000$. For each method, we use $4500$ trajectories.  The reported KL and $\chi^2$ divergences are computed from normalized histogram bin masses
against the normalized grid approximation of $\pi_G$, using histogram spacing $\Delta x=0.05$.  The convergence curves use pooled histograms over windows of $50$ saved frames, corresponding to a physical window length of $50$ time units.
\begin{figure}[!htb]
    \centering
    \includegraphics[width=0.72\linewidth]{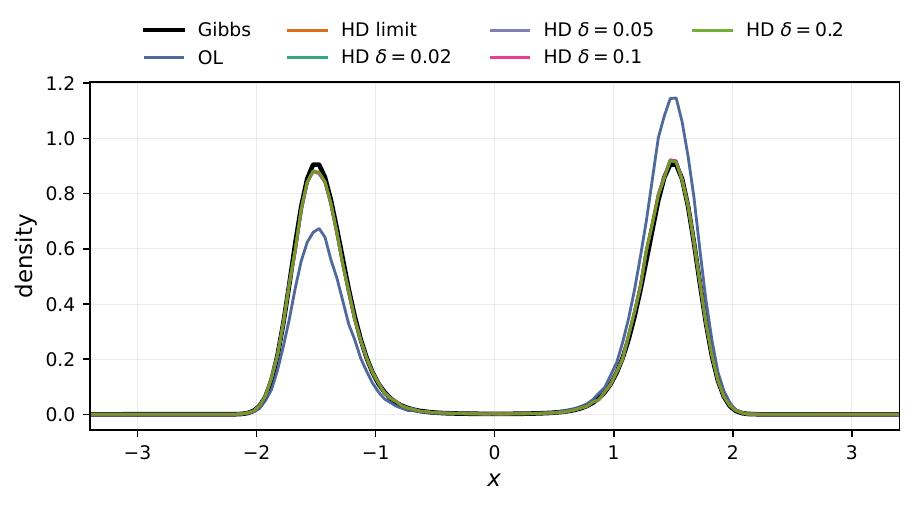}
    \caption{Final pooled empirical densities compared with the Gibbs density in Example I. The HD and regularized HD variants match the target density closely, while OL retains a visible residual bias at $T=2000$.}
    \label{fig:regularized-hd-1d-densities}
\end{figure}

We show in Figure~\ref{fig:regularized-hd-1d-densities} the empirical densities given by the different schemes as well as the corresponding Gibbs density. At the final time of the simulation, the HD and $\delta$-HD produced the densities that are closer to the Gibbs than the overdamped Langevin. Figure~\ref{fig:regularized-hd-1d-convergence}
shows the KL convergence history.  The regularized HD curves are nearly indistinguishable from the limiting HD curve at this time step and sample size.  After a short initial transient, the HD curves overtake OL around $t=83$ and remain below OL for the rest of the simulation. In particular, all HD variants reach $\mathrm{KL}\le 10^{-2}$ by time approximately $810$, while OL does not reach this accuracy by the final time $T=2000$.  At the final time, OL has KL divergence $3.53\times 10^{-2}$, whereas the HD variants have final KL divergences between $3.88\times 10^{-4}$ and $4.51\times 10^{-4}$.
\begin{figure}[!htb]
    \centering
    \includegraphics[width=0.68\linewidth]{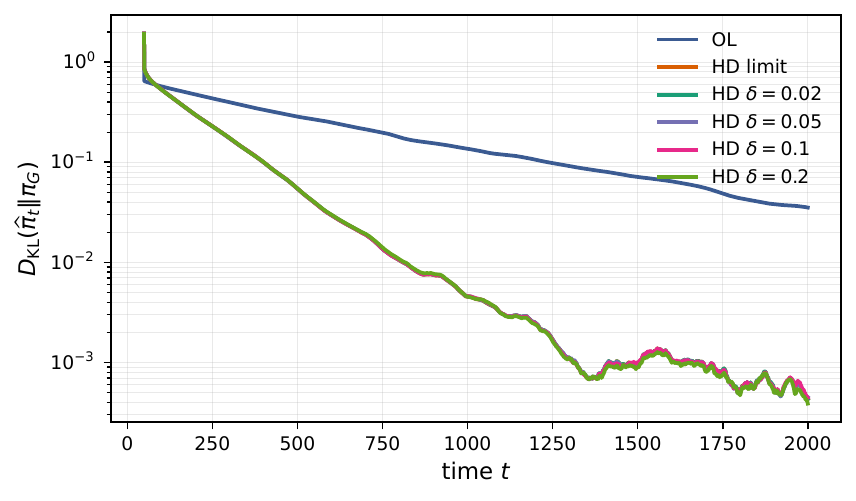}
    \caption{KL convergence histories for OL, limiting HD, and regularized HD
    with $\delta\in\{0.02,0.05,0.10,0.20\}$ in Example I.  The regularized HD curves
    closely track the limiting HD curve and remain uniformly below OL after
    the short initial transient.}
    \label{fig:regularized-hd-1d-convergence}
\end{figure}

\begin{table}[!htb]
    \centering
    \begin{tabular}{lccccc}
        \toprule
        Method & Final KL & Final $\chi^2$
        & $t_{\mathrm{KL}\le 10^{-1}}$
        & $t_{\mathrm{KL}\le 5\cdot10^{-2}}$
        & $t_{\mathrm{KL}\le 10^{-2}}$ \\
        \midrule
        OL & $0.0353$ & $0.0698$ & $1237$ & $1741$ & -- \\
        HD limit & $0.000430$ & $0.000954$ & $401$ & $513$ & $810$ \\
        HD, $\delta=0.02$ & $0.000429$ & $0.000951$ & $402$ & $513$ & $811$ \\
        HD, $\delta=0.05$ & $0.000451$ & $0.000997$ & $401$ & $513$ & $812$ \\
        HD, $\delta=0.10$ & $0.000444$ & $0.000979$ & $402$ & $513$ & $811$ \\
        HD, $\delta=0.20$ & $0.000388$ & $0.000869$ & $402$ & $514$ & $817$ \\
        \bottomrule
    \end{tabular}
    \caption{Benchmark summary for OL, limiting HD, and $C^2$-regularized HD in Example I.
    }
    \label{tab:regularized-hd-1d-benchmark}
\end{table}
The comparison across $\delta$, in both Figure~\ref{fig:regularized-hd-1d-convergence} and Table~\ref{tab:regularized-hd-1d-benchmark}, indicates that the observed acceleration persists under $C^2$ regularization of the HD interface. The regularized dynamics replace the abrupt switch by a thin transition layer in which both $a_\delta$ and the correction drift $b_\delta$ vary smoothly, but the observed convergence is essentially unchanged for the tested widths. This indicates that the main effect comes from flattening the weighted conductance $a(\bx)\pi(\bx)$ inside the hybridized region. The drift-free interior dynamics removes the usual energetic bottleneck in the Dirichlet form, while the exterior OL dynamics preserves confinement and pulls trajectories back toward the target region.
The small differences among the $\delta$-curves are comparable to the Monte Carlo resolution of the pooled histogram estimator. In this example, regularization improves the smoothness of the numerical dynamics without degrading the sampling advantage of HD over OL.

We also tested the fully derivative-free dynamics obtained by applying the interior HD diffusion coefficient on the entire real line,
\[
X_n+\sqrt{2\eta\,\eps\exp\left((F(X_n)-F_0)/\eps\right)}\,\xi_n.
\]
This experiment is not included as a density plot because it is numerically unstable at the same step size and initialization used above.  Starting from $N(3,0.01^2)$, all $4500$ trajectories leave the plotting window $[-4,4]$ after the a few time steps.  Thus the pure derivative-free dynamics does not provide a meaningful sampling baseline for this setup.  The observation instead reinforces the role of the hybrid switch: large diffusion is useful inside the low-energy region, but outside this region the OL part is needed to provide confinement.
\begin{figure}[!htb]
    \centering
    \includegraphics[width=0.6\linewidth]{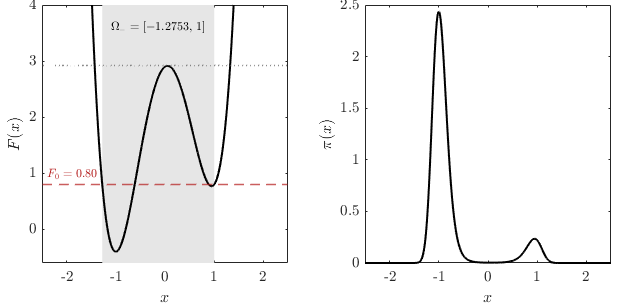} 
    \caption{The asymmetric double potential $F$ and the corresponding Gibbs distribution (with $\eps=0.5$) in Example II.}
    \label{FIG:1D Exp2 Setup}
\end{figure}

\paragraph{Example II (uniform initialization).} In the second numerical example, we consider the asymmetric potential given by (see Figure~\ref{FIG:1D Exp2 Setup}):
\[
    F(x)=3(x^2-1)^2+0.3(x+1)^2-0.4\,.
\]
We choose the interior region for the derivative-free dynamics as $\Omega_-:=(-1.2753, 1)$. This means that $F_0=0.8$ for the potential. Note that $\Omega_-$ is not symmetric with respect to $x=0$ due to the non-symmetry of $F$.
\begin{figure}[!htb]
    \centering
    \includegraphics[width=0.98\linewidth]{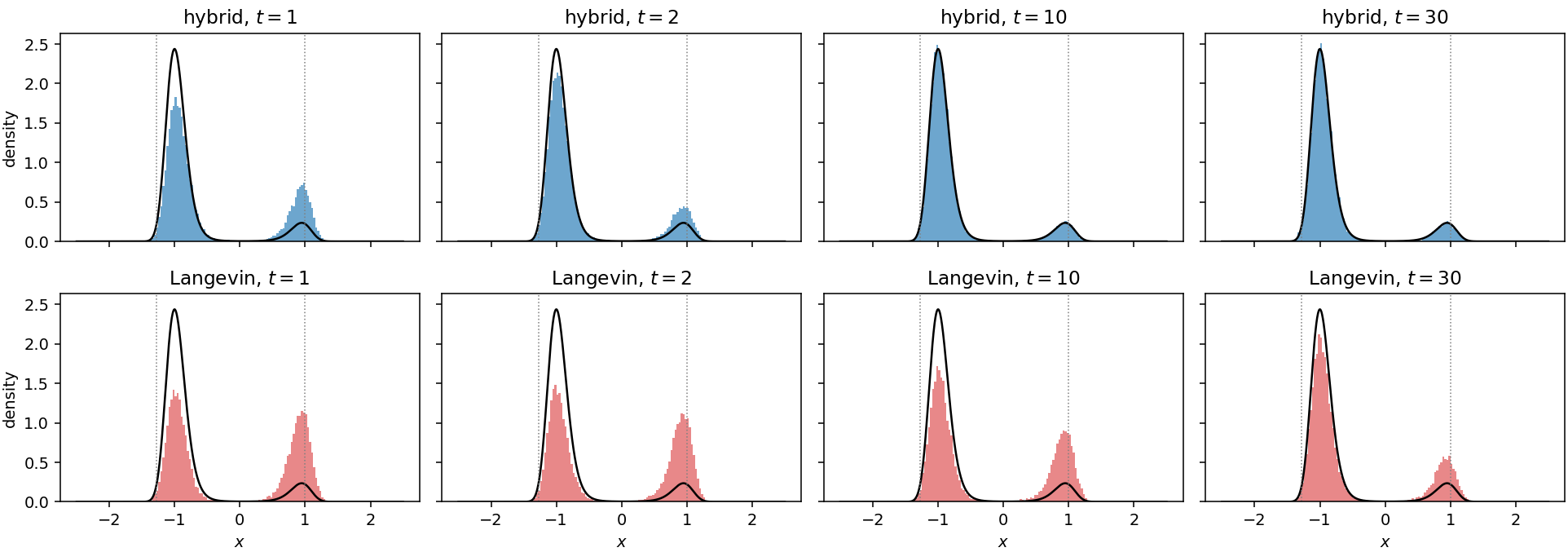} 
    \caption{Snapshots of the empirical densities given by the hybrid dynamics and the overdamped Langevin at $t=1, 2, 10, 30$ (against the Gibbs density) in Example II.}
    \label{FIG:1D Exp2 Snapshots}
\end{figure}

We start the simulations from a uniform density $\cU([-2.5, 2.5])$. We compare here only the hybrid dynamics and the overdamped Langevin dynamics. Figure~\ref{FIG:1D Exp2 Snapshots} shows the comparison of the empirical snapshots of the densities $\wh\rho^N(t,\cdot)$ at $t=1, 2, 10$ and $30$ from $N=2\times 10^3$ trajectories. Both samplers reproduce the two-peaked structure of $\pi$ fairly well by the time $t=30$. However, we see a clear advantage of the hybrid dynamics in this case from the evolution of the snapshots and the $\KL$ divergence $\KL(\wh\rho^N(t,\cdot)\|\pi)$ in Figure~\ref{FIG:1D Exp2 KL}. The hybrid dynamics quickly redistribute the extra mass from the smaller peak at $x=1$ to the larger peak at $x=-1$, resulting in faster convergence than the overdamped Langevin dynamics.

\begin{figure}[!htb]
    \centering   \includegraphics[width=0.5\linewidth,height=0.34\textwidth]{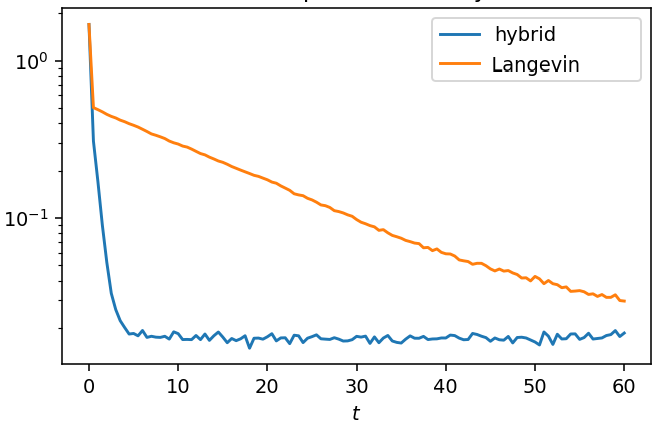} 
    \caption{Evolution of $\KL$ divergence for hybrid dynamics and overdamped Langevin in Example II.}
    \label{FIG:1D Exp2 KL}
\end{figure}

\subsection{2D sampling examples}
\label{subsec:2d-radial-benchmark}

We next consider a two-dimensional example designed to test whether HD can rapidly redistribute probability mass among several metastable radial wells.

\paragraph{Example III (radially symmetric potential).} Let $r=|\bx|$ and $q=r^2=|\bx|^2$.  We define
\[
\theta(q)=6\pi(q-0.04),
\]
and
\[
G(q)
=0.90\left(1-\cos\theta(q)\right)
+0.45\left(1-\cos(2\theta(q))\right)
+0.35\sin^2\theta(q)
+0.18\sin(3\theta(q)).
\]
With
$\delta=0.08$, set
\[
\psi_\delta(s)=
\begin{cases}
0, & s\le 0,\\[0.5ex]
\delta^2\left(t^5-3t^4+3t^3\right),
\qquad t=s/\delta, & 0<s<\delta,\\[0.5ex]
s^2, & s\ge \delta.
\end{cases}
\]
The radial profile is
\[
V_{\rm rad}(r)=G(r^2)+80\,\psi_\delta(r-1),
\]
and the two-dimensional potential is
\[
V(\bx)=V_{\rm rad}(|\bx|)
=G(|\bx|^2)+80\,\psi_\delta(|\bx|-1).
\]
Thus $V\in \cC^2(\mathbb R^2)$, has several oscillatory radial wells in the unit disk, and grows quadratically for $|\bx|\ge 1+\delta$.  The target Gibbs distribution is
\[
\pi_\eps(\bx)=Z_\eps^{-1}\exp(-V(\bx)/\eps),
\qquad \eps=0.35 .
\]
Figure~\ref{fig:2d-radial-setup} shows the radial profile, the HD level-set region, and the Gibbs density.
\begin{figure}[!htb]
\centering
\includegraphics[width=\textwidth]{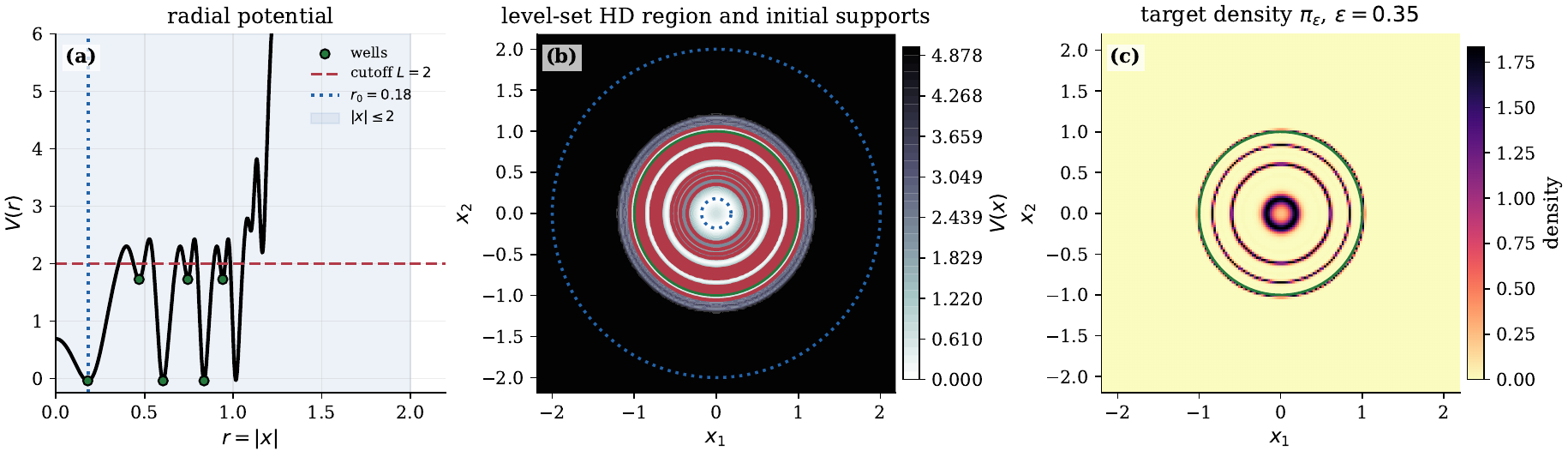}
\caption{Setup for Example III. Panel (a) shows the radial potential $V(r)$, the HD cutoff $F_0=2$, the near-well initial radius $r_0=0.18$ used in $\mu_0^{\mathrm{well}}$, and the support radius of the uniform-disk initialization used in $\mu_0^{\mathrm{disk}}$. Panel (b) shows the level-set region $V(\bx)<2$ where HD uses diffusion-only dynamics.  Panel (c) shows the density of the target Gibbs distribution $\pi_\eps\propto\exp(-V/\eps)$, with $\eps=0.35$.}
\label{fig:2d-radial-setup}
\end{figure}

Here the HD interior diffusion coefficient is $D(\bx)=\eps\exp((V(\bx)-F_0)/\eps)$, so that the hybrid update uses increment $\sqrt{2\eta D(X_n)}\xi_n$ on $\{V<F_0\}$ and the OL update on $\{V\ge F_0\}$. In all experiments below $F_0=2$, so $D(\bx)=\eps$ continuously on the interface $V(\bx)=F_0$.  Both OL and HD methods use step size $\eta=10^{-3}$, final time $T=200$, and $16000$ particles per run.  Reported $\KL$ values are computed from normalized histogram cell masses against the normalized grid approximation of $\pi_\eps$. We use six independent runs and also report the KL of the pooled histogram.
\begin{table}[!htb]
\centering
\begin{tabular}{llcccc}
\toprule
initial law & method & $t=40$ & $t=80$ & $t=120$ & $t=200$ \\
\midrule
$\mu_0^{\mathrm{well}}$ & OL & 1.210 & 1.216 & 1.213 & 1.211 \\
$\mu_0^{\mathrm{well}}$ & HD, $F_0=2$ & 0.141 & 0.0813 & 0.0737 & 0.0737 \\
$\mu_0^{\mathrm{disk}}$ & OL & 1.211 & 1.217 & 1.216 & 1.211 \\
$\mu_0^{\mathrm{disk}}$ & HD, $F_0=2$ & 0.0760 & 0.0748 & 0.0730 & 0.0744 \\
\bottomrule
\end{tabular}
\caption{KL divergence
$D_{\mathrm{KL}}(\widehat\pi_t\Vert\pi_\eps)$ for the two initial distributions in Example III.
HD rapidly approaches the Gibbs radial shell structure, while OL remains localized near the inner well for both initial distributions.}
\label{tab:2d-radial-kl}
\end{table}

We test two initial distributions:
\[
\mu_0^{\mathrm{well}}:\quad
X_0=0.18(\cos\Theta,\sin\Theta)+0.02Z,
\qquad
\Theta\sim\mathrm{Unif}(0,2\pi),\quad Z\sim N(0,I_2),
\]
and
\[
\mu_0^{\mathrm{disk}}:\quad
X_0\sim \mathrm{Unif}\{\bx\in\mathbb R^2:\ |\bx|\le 2\}.
\]
The first initial distribution is concentrated near the innermost well, while the second one places mass well outside the target region.

\begin{figure}[!htb]
\centering
\includegraphics[width=\textwidth]{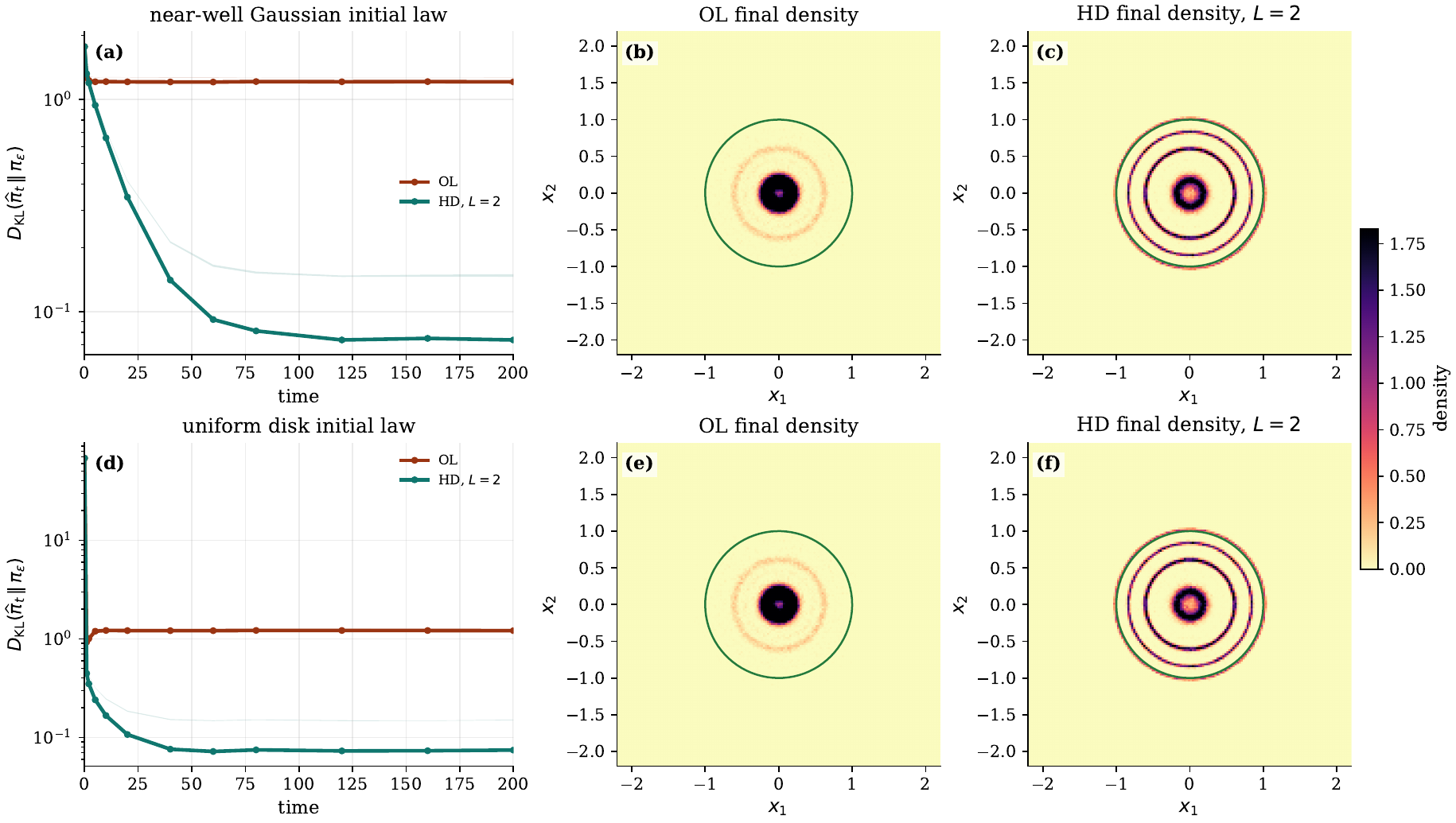}
\caption{Sampling comparison for two initial distributions in Example III.  Panels (a)--(c) use the near-well Gaussian initial distribution $\mu_0^{\mathrm{well}}$, and panels (d)--(f) use the broad uniform-disk initial distribution $\mu_0^{\mathrm{disk}}$. The left column shows pooled KL convergence curves, with shaded bands indicating one standard error across six runs.  The middle and right columns show the final pooled empirical densities for OL and HD, respectively.}
\label{fig:2d-radial-results}
\end{figure}
The results are shown in Figure~\ref{fig:2d-radial-results} and summarized in
Table~\ref{tab:2d-radial-kl}.  For the near-well initial distribution $\mu_0^{\mathrm{well}}$, HD reduces the pooled KL divergence from order one to $7.37\times10^{-2}$ by final time, while OL remains near $1.21$.  For the broad uniform-disk initialization $\mu_0^{\mathrm{disk}}$, HD reaches essentially the same accuracy: the best recorded pooled KL is $7.21\times10^{-2}$ at $t=60$, and the final pooled KL is $7.44\times10^{-2}$.  In contrast, OL again remains near $1.21$.  The final density plots show that OL quickly contracts mass toward the inner well and ignore the outer target rings, whereas HD equilibrates the
radial shell masses much more effectively.


\paragraph{Example IV (anisotropic potential).} We next deform the radial example into a genuinely anisotropic two-dimensional example by replacing the Euclidean radius with an elliptic radius. This tests whether HD still rapidly redistributes probability mass when the metastable wells are organized along elliptic, rather than circular, level sets. Define 
\[ s(\bx)=\sqrt{\left(\frac{x_1}{1.45}\right)^2+ \left(\frac{x_2}{0.75}\right)^2}, \qquad \bx=(x_1,x_2)\in\mathbb R^2, 
\] 
and set $g(\bx)=s(\bx)^2$. Using the same oscillatory profile $G$ and the same $C^2$ confining regularization $\psi_\delta$ from Example~III, with $\delta=0.08$, we define \[ V(\bx)=G(g(\bx))+80\,\psi_\delta(s(\bx)-1). \] Thus $V\in\mathcal C^2(\mathbb R^2)$, has several oscillatory wells along elliptic shells, and grows quadratically as $s(\bx)\to\infty$. The target Gibbs distribution is 
\[ \pi_\eps(\bx)=Z_\eps^{-1}\exp(-V(\bx)/\eps), \qquad \eps=0.35 . 
\]
Figure~\ref{fig:2d-elliptic-setup} shows the elliptic-radius profile, the HD level-set region, and the Gibbs density.
\begin{figure}[!htb]
\centering
\includegraphics[width=\textwidth]{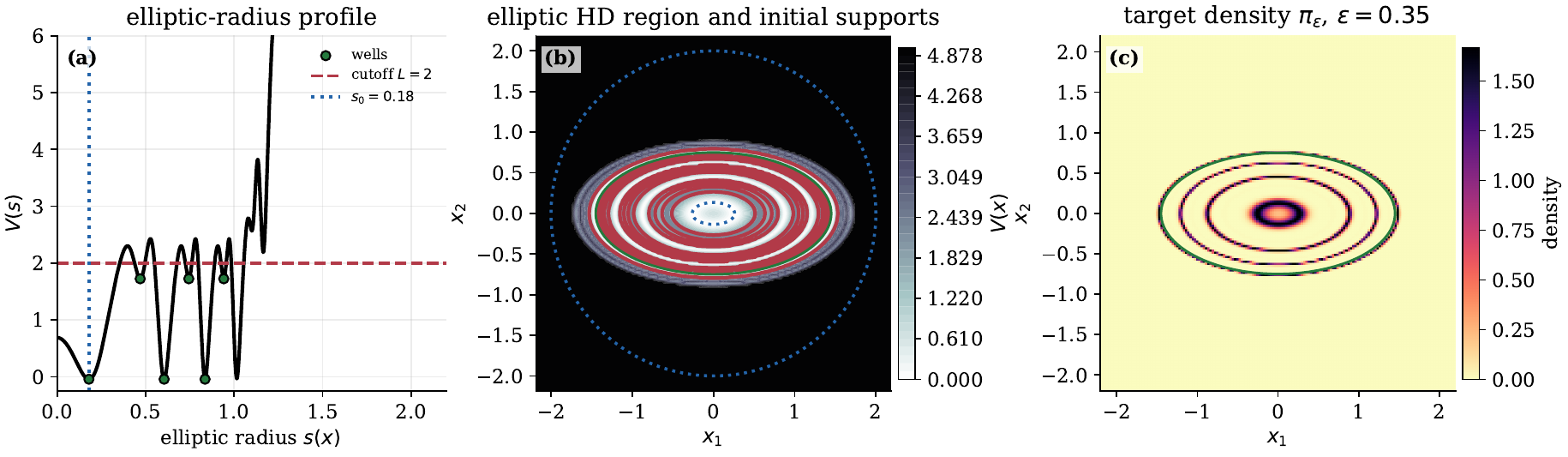}
\caption{Setup for Example IV. Panel (a) shows the potential profile as a function of the elliptic radius $s(\bx)$, the HD cutoff $F_0=2$, and the near-well elliptic initialization radius $s_0=0.18$.  Panel (b) shows the sublevel region $V(\bx)<2$, where HD uses diffusion-only dynamics.  Panel (c) shows the density of the target Gibbs distribution $\pi_\eps\propto\exp(-V/\eps)$, with $\eps=0.35$.}
\label{fig:2d-elliptic-setup}
\end{figure}

We use the same OL and HD discretizations as in Section~\ref{subsec:2d-radial-benchmark}. In all experiments below $F_0=2$, so
\[
D(\bx)=\eps\exp((V(\bx)-F_0)/\eps)
\]
satisfies $D(\bx)=\eps$ continuously on the interface $V(\bx)=F_0$.  Both methods use step size $\eta=10^{-3}$, final time $T=200$, and $16000$ particles per run.  $\KL$ values are computed from normalized histogram cell masses against the normalized grid approximation of $\pi_\eps$. We use six independent runs and report the $\KL$ of the pooled histogram.

We test two initial distributions:
\[
\mu_0^{\mathrm{ell}}:\quad
X_0=\bigl(1.45\cdot 0.18\cos\Theta,\;0.75\cdot 0.18\sin\Theta\bigr)+0.02Z,
\qquad
\Theta\sim\mathrm{Unif}(0,2\pi),\quad Z\sim N(0,I_2),
\]
and
\[
\mu_0^{\mathrm{disk}}:\quad
X_0\sim \mathrm{Unif}\{\bx\in\mathbb R^2:\ |\bx|\le 2\}.
\]
The first initial distribution is concentrated near the innermost elliptic well, while the second one places mass well outside the target region.
\begin{table}[!htb]
\centering
\begin{tabular}{llcccc}
\toprule
initial law & method & $t=40$ & $t=80$ & $t=120$ & $t=200$ \\
\midrule
$\mu_0^{\mathrm{ell}}$ & OL & 1.271 & 1.271 & 1.273 & 1.270 \\
$\mu_0^{\mathrm{ell}}$ & HD, $F_0=2$ & 0.231 & 0.129 & 0.115 & 0.109 \\
$\mu_0^{\mathrm{disk}}$ & OL & 1.271 & 1.273 & 1.274 & 1.273 \\
$\mu_0^{\mathrm{disk}}$ & HD, $F_0=2$ & 0.114 & 0.113 & 0.110 & 0.109 \\
\bottomrule
\end{tabular}
\caption{KL divergence
$D_{\mathrm{KL}}(\widehat\pi_t\Vert\pi_\eps)$ for Example IV. HD rapidly approaches the Gibbs elliptic shell structure, while OL remains
localized near the inner well for both initial distributions.}
\label{tab:2d-elliptic-kl}
\end{table}

\begin{figure}[!htb]
\centering
\includegraphics[width=\textwidth]{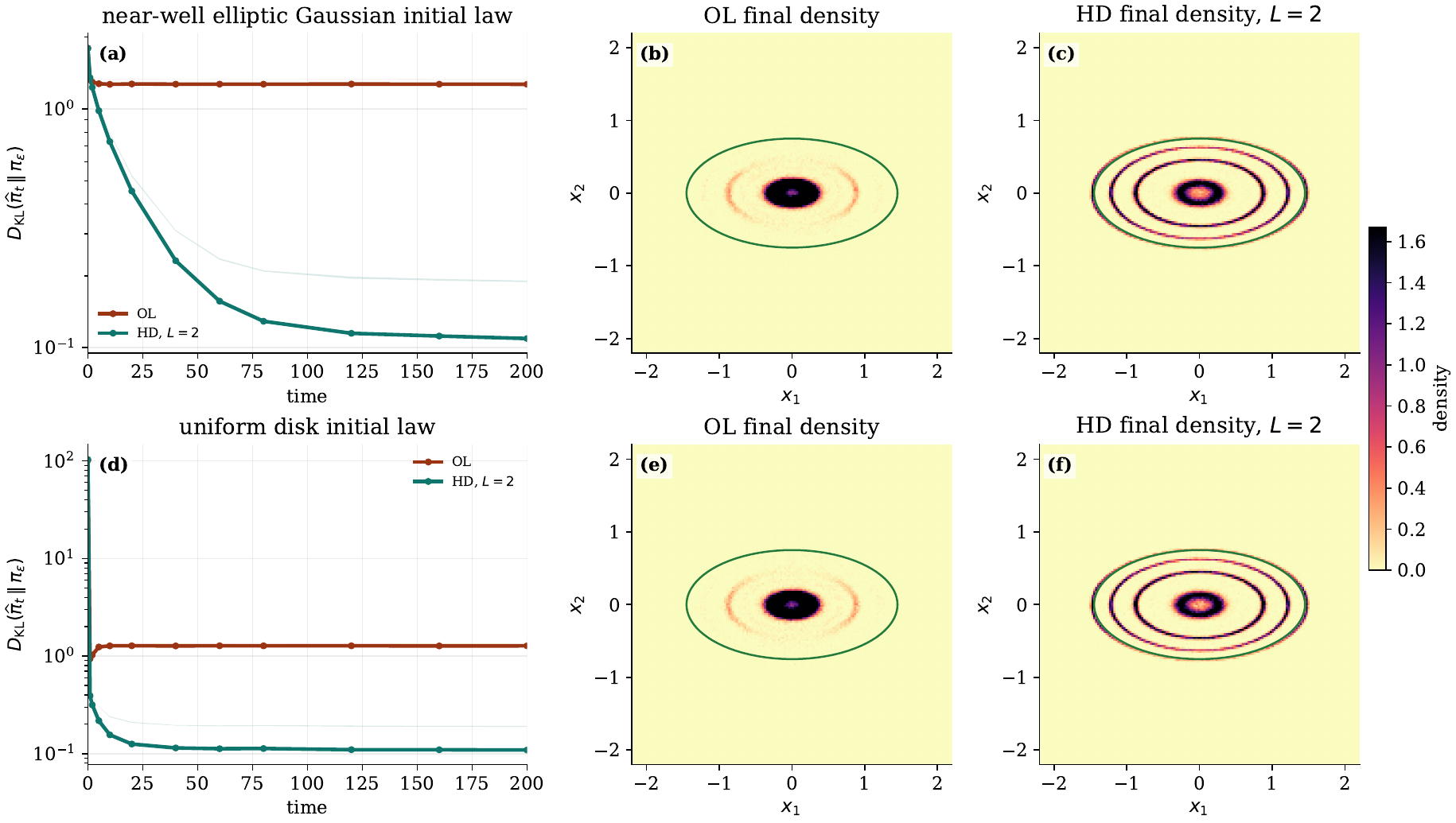}
\caption{Sampling comparison for two initial distributions in Example IV.  Panels (a)--(c) use the near-well elliptic Gaussian initial distribution $\mu_0^{\mathrm{ell}}$; panels (d)--(f) use the broad uniform-disk initial distribution $\mu_0^{\mathrm{disk}}$.  The left column shows pooled KL convergence curves, with shaded bands indicating one standard error across six runs.  The middle and right columns show the final pooled empirical densities for OL and HD, respectively.}
\label{fig:2d-elliptic-results}
\end{figure}
The results are shown in Figure~\ref{fig:2d-elliptic-results} and summarized in Table~\ref{tab:2d-elliptic-kl}. For the near-well-elliptic initialization $\mu_0^{\mathrm{ell}}$, HD reduces the pooled KL divergence to $1.09\times 10^{-1}$ by the final time, while OL remains near $1.27$. For the broad uniform-disk initialization $\mu_0^{\mathrm{disk}}$, HD again achieves essentially the same accuracy, with a final pooled KL of $1.09\times 10^{-1}$, whereas OL remains near $1.27$. The final density plots show the same behaviors of the two dynamics as in the radial example: OL moves mass toward the innermost well and fails to put enough mass on the outer elliptic shells. At the same time, HD equilibrates mass across all shell structures.

\section{Concluding remarks}
\label{SEC:Concl}

This work introduces a hybrid dynamics framework for sampling from Gibbs distributions by combining two distinct stochastic dynamics in different regions of the state space. The construction is designed so that the resulting process preserves the target Gibbs distribution while modifying the dynamics' effective geometry in regions where metastability is expected to dominate.

We developed a rigorous analysis of the hybrid dynamics through a regularization approach. We showed that the hybrid dynamics converges exponentially fast to the Gibbs distribution. We also analyzed the metastability properties of the hybrid dynamics in a radially symmetric landscape. We showed that the transition mechanism between metastable states is fundamentally altered. While overdamped Langevin dynamics is governed by the saddle barrier, the hybrid dynamics replaces it with the interface's switching level. As a consequence, an appropriate choice of the switching surface leads to exponentially faster transition times, demonstrating a clear advantage of the hybrid construction in multimodal landscapes.

A central question for the future is how to choose the switching interface in a systematic and possibly adaptive manner. In complex high-dimensional landscapes, it would be desirable to identify regions where the hybridization yields the greatest improvement, potentially using data-driven or learning-based approaches. Also, the present framework combines overdamped Langevin dynamics with a derivative-free diffusion. It would be of interest to extend the hybrid approach to incorporate other dynamics, such as underdamped Langevin processes, preconditioned diffusions, or nonreversible dynamics, and to understand how these choices affect both convergence rates and metastability.



\section*{Data availability statements}
Data sets generated during the current study are available from the corresponding author on reasonable request.


\section*{Declarations}
The authors declare no competing interests. This work is partially supported by the National Science Foundation through grants DMS-2208504 (BE), DMS-1937254 (KR), DMS-2309802 (KR), and DMS-2409855 (YY).  YY~was also partially supported by Office of Naval Research through grant N00014-24-1-2088.

{\small

\begin{thebibliography}{10}

\bibitem{AmGiSa-Book08}
{\sc L.~Ambrosio, N.~Gigli, and G.~Savar{\'e}}, {\em Gradient flows: in metric
  spaces and in the space of probability measures}, Springer Science \&
  Business Media, second~ed., 2008.

\bibitem{AnLi-AOS21}
{\sc C.~Andrieu and S.~Livingstone}, {\em Peskun--{T}ierney ordering for
  {M}arkovian {M}onte {C}arlo: beyond the reversible scenario}, Ann. Statist.,
  49 (2021), pp.~1958--1981.

\bibitem{ArMaToUn-CPDE01}
{\sc A.~Arnold, P.~Markowich, G.~Toscani, and A.~Unterreiter}, {\em On convex
  {Sobolev} inequalities and the rate of convergence to equilibrium for
  {Fokker-Planck} type equations}, Commun. PDEs, 26 (2001), pp.~43--100.

\bibitem{Ar-BAMS-67}
{\sc D.~G. Aronson}, {\em Bounds for the fundamental solution of a parabolic
  equation}, Bull. Amer. Math. Soc., 73 (1967), pp.~890--896.

\bibitem{BaEm-SP85}
{\sc D.~Bakry and M.~\'Emery}, {\em Diffusions hypercontractives}, in
  S\'eminaire de Probabilit\'es, XIX, 1983/84, vol.~1123 of Lecture Notes in
  Mathematics, Springer, 1985, pp.~177--206.

\bibitem{BaGeLe-Book14}
{\sc D.~Bakry, I.~Gentil, and M.~Ledoux}, {\em Analysis and Geometry of
  {Markov} Diffusion Operators}, Springer, 2014.

\bibitem{BeHeDoJa-arXiv19}
{\sc E.~Bernton, J.~Heng, A.~Doucet, and P.~E. Jacob}, {\em {S}chr{\"o}dinger
  bridge samplers}, arXiv preprint arXiv:1912.13170,  (2019).

\bibitem{BoEbZi-AAP20}
{\sc N.~Bou-Rabee, A.~Eberle, and R.~Zimmer}, {\em Coupling and convergence for
  {H}amiltonian {M}onte {C}arlo}, Ann. Appl. Probab., 30 (2020),
  pp.~1209--1250.

\bibitem{BoVoDo-JASA18}
{\sc A.~Bouchard-C{\^o}t{\'e}, S.~J. Vollmer, and A.~Doucet}, {\em The bouncy
  particle sampler: A non-reversible rejection-free {M}arkov chain {M}onte
  {C}arlo method}, J. Amer. Statist. Assoc., 113 (2018), pp.~855--867.

\bibitem{BrGeJoMe-Book11}
{\sc S.~Brooks, A.~Gelman, G.~Jones, and X.-L. Meng}, {\em {Handbook of Markov
  chain Monte Carlo}}, CRC press, 2011.

\bibitem{CaLuWa-ARMA23}
{\sc Y.~Cao, J.~Lu, and L.~Wang}, {\em On explicit {$L^2$}-convergence rate
  estimate for underdamped {L}angevin dynamics}, Archive for Rational Mechanics
  and Analysis, 247 (2023), p.~Art. 90.

\bibitem{CaJuMaToUn-MM01}
{\sc J.~A. Carrillo, A.~J{\"u}ngel, P.~A. Markowich, G.~Toscani, and
  A.~Unterreiter}, {\em Entropy dissipation methods for degenerate parabolic
  problems and generalized {Sobolev} inequalities}, Monatshefte f{\"u}r
  Mathematik, 133 (2001), pp.~1--82.

\bibitem{ChChBaJo-PMLR18}
{\sc X.~Cheng, N.~S. Chatterji, P.~L. Bartlett, and M.~I. Jordan}, {\em
  Underdamped langevin mcmc: A non-asymptotic analysis}, in Conference on
  learning theory, PMLR, 2018, pp.~300--323.

\bibitem{ChErLiShZh-FCM25}
{\sc S.~Chewi, M.~A. Erdogdu, M.~Li, R.~Shen, and M.~S. Zhang}, {\em Analysis
  of {Langevin} {Monte Carlo} from {Poincare} to {log-Sobolev}}, Foundations of
  Computational Mathematics, 25 (2025), pp.~1345--1395.

\bibitem{DaRi-Bernoulli20}
{\sc A.~S. Dalalyan and L.~Riou-Durand}, {\em On sampling from a log-concave
  density using kinetic {Langevin} diffusions}, Bernoulli, 26 (2020),
  pp.~1956--1988.

\bibitem{DuEl-Book97}
{\sc P.~Dupuis and R.~S. Ellis}, {\em A Weak Convergence Approach to the Theory
  of Large Deviations}, Wiley, New York, 1997.

\bibitem{DuLiPlDo-MMS12}
{\sc P.~Dupuis, Y.~Liu, N.~Plattner, and J.~D. Doll}, {\em On the infinite
  swapping limit for parallel tempering}, Multiscale Modeling \& Simulation, 10
  (2012), pp.~986--1022.

\bibitem{DuMo-AAP17}
{\sc A.~Durmus and {\'E}.~Moulines}, {\em Nonasymptotic convergence analysis
  for the unadjusted {L}angevin algorithm}, The Annals of Applied Probability,
  27 (2017), pp.~1551--1587.

\bibitem{EaDe-PCCP05}
{\sc D.~J. Earl and M.~W. Deem}, {\em Parallel tempering: Theory, applications,
  and new perspectives}, Phys. Chem. Chem. Phys., 7 (2005), pp.~3910--3916.

\bibitem{EbGuZi-AOP19}
{\sc A.~Eberle, A.~Guillin, and R.~Zimmer}, {\em Couplings and quantitative
  contraction rates for {L}angevin dynamics}, Ann. Probab., 47 (2019),
  pp.~1982--2010.

\bibitem{EnReYa-CAMC24}
{\sc B.~Engquist, K.~Ren, and Y.~Yang}, {\em Adaptive state-dependent diffusion
  for derivative-free optimization}, Commun. Appl. Math. Comput., 6 (2024),
  pp.~1241--1269.

\bibitem{EnReYa-arXiv24}
{\sc B.~Engquist, K.~Ren, and Y.~Yang}, {\em Sampling with adaptive variance
  for multimodal distributions}, arXiv:2411.15220,  (2024).

\bibitem{ErHo-PMLR21}
{\sc M.~A. Erdogdu and R.~Hosseinzadeh}, {\em On the convergence of {Langevin}
  monte carlo: The interplay between tail growth and smoothness}, in Conference
  on Learning Theory, PMLR, 2021, pp.~1776--1822.

\bibitem{EvGa-Book15}
{\sc L.~C. Evans and R.~F. Gariepy}, {\em Measure theory and fine properties of
  functions}, Textbooks in Mathematics, CRC Press, revised~ed., 2015.

\bibitem{FaSaSk-JCP14}
{\sc Y.~Fang, J.~M. Sanz-Serna, and R.~D. Skeel}, {\em Compressible generalized
  hybrid {M}onte {C}arlo}, J. Chem. Phys., 140 (2014), p.~174108.

\bibitem{Friedman-Book08}
{\sc A.~Friedman}, {\em Partial Differential Equations of Parabolic Type},
  Courier Dover Publications, 2008.

\bibitem{GeLeRi-NeurIPS18}
{\sc R.~Ge, H.~Lee, and A.~Risteski}, {\em Simulated tempering {L}angevin
  {M}onte {C}arlo with multimodal distributions}, in Advances in Neural
  Information Processing Systems (NeurIPS), 2018.

\bibitem{GiCa-JRSS11}
{\sc M.~Girolami and B.~Calderhead}, {\em Riemann manifold {Langevin} and
  {Hamiltonian} {Monte Carlo} methods}, Journal of the Royal Statistical
  Society Series B: Statistical Methodology, 73 (2011), pp.~123--214.

\bibitem{HaScZh-Nonlinearity16}
{\sc C.~Hartmann, C.~Sch{\"u}tte, and W.~Zhang}, {\em Model reduction
  algorithms for optimal control and importance sampling of diffusions},
  Nonlinearity, 29 (2016), pp.~2298--2326.

\bibitem{HoSt-JSP87}
{\sc R.~Holley and D.~W. Stroock}, {\em Logarithmic {Sobolev} inequalities and
  stochastic {Ising} models}, J. Stat. Phys., 46 (1987), pp.~1159--1194.

\bibitem{JoKiOt-SIAM98}
{\sc R.~Jordan, D.~Kinderlehrer, and F.~Otto}, {\em The variational formulation
  of the {F}okker--{P}lanck equation}, SIAM Journal on Mathematical Analysis,
  29 (1998), pp.~1--17.

\bibitem{Jungel-Book16}
{\sc A.~J{\"u}ngel}, {\em Entropy Methods for Diffusive Partial Differential
  Equations}, Springer, 2016.

\bibitem{LaSoUr-Book68}
{\sc O.~A. Ladyzhenskaya, V.~A. Solonnikov, and N.~N. Ural\'tseva}, {\em Linear
  and Quasi-linear Equations of Parabolic Type}, American Mathematical Society,
  1968.

\bibitem{LaRoRo-AAP13}
{\sc K.~{\L}atuszy{\'n}ski, G.~O. Roberts, and J.~S. Rosenthal}, {\em Adaptive
  {G}ibbs samplers and related {MCMC} methods}, Ann. Appl. Probab., 23 (2013),
  pp.~66--98.

\bibitem{Ledoux-SP04}
{\sc M.~Ledoux}, {\em Logarithmic {Sobolev} inequalities for unbounded spin
  systems revisited}, in S{\'e}minaire de Probabilit{\'e}s XXXV, Springer,
  2004, pp.~167--194.

\bibitem{LeRiGe-NIPS18}
{\sc H.~Lee, A.~Risteski, and R.~Ge}, {\em Beyond log-concavity: Provable
  guarantees for sampling multi-modal distributions using simulated tempering
  {Langevin} {Monte Carlo}}, Advances in Neural Information Processing Systems,
  31 (2018).

\bibitem{LeMaSt-IMAJNA16}
{\sc B.~Leimkuhler, C.~Matthews, and G.~Stoltz}, {\em The computation of
  averages from equilibrium and nonequilibrium {L}angevin molecular dynamics},
  IMA J. Numer. Anal., 36 (2016), pp.~13--79.

\bibitem{LeRoSt-Book10}
{\sc T.~Lelievre, M.~Rousset, and G.~Stoltz}, {\em Free Energy Computations: a
  Mathematical Perspective}, World Scientific, 2010.

\bibitem{LiChCaCa-AAAI16}
{\sc C.~Li, C.~Chen, D.~Carlson, and L.~Carin}, {\em Preconditioned stochastic
  gradient {Langevin} dynamics for deep neural networks}, in Proceedings of the
  AAAI Conference on Artificial Intelligence, vol.~30, 2016.

\bibitem{LiWa-NeurIPS16}
{\sc Q.~Liu and D.~Wang}, {\em {S}tein variational gradient descent: A general
  purpose {B}ayesian inference algorithm}, in Advances in Neural Information
  Processing Systems (NeurIPS), 2016.

\bibitem{LuSlWa-Nonlinearity23}
{\sc Y.~Lu, D.~Slep{\v c}ev, and L.~Wang}, {\em Birth--death dynamics for
  sampling: global convergence, approximations and their asymptotics},
  Nonlinearity, 36 (2023), pp.~5731--5772.

\bibitem{MaChJiFlJo-Bernoulli21}
{\sc Y.-A. Ma, Y.~Chen, C.~Jin, N.~Flammarion, and M.~I. Jordan}, {\em Is there
  an analog of {N}esterov acceleration for gradient-based {MCMC}?}, Bernoulli,
  27 (2021), pp.~1942--1992.

\bibitem{MaPa-EL92}
{\sc E.~Marinari and G.~Parisi}, {\em Simulated tempering: A new {M}onte
  {C}arlo scheme}, Europhysics Letters, 19 (1992), pp.~451--458.

\bibitem{MaVi-MC00}
{\sc P.~A. Markowich and C.~Villani}, {\em On the trend to equilibrium for the
  fokker-planck equation: an interplay between physics and functional
  analysis}, Mat. Contemp, 19 (2000), pp.~1--29.

\bibitem{Monmarche-EJS21}
{\sc P.~Monmarch{\'e}}, {\em High-dimensional {MCMC} with a standard splitting
  scheme for the underdamped {L}angevin diffusion}, Electron. J. Stat., 15
  (2021), pp.~4117--4166.

\bibitem{Pavliotis-Book14}
{\sc G.~A. Pavliotis}, {\em Stochastic Processes and Applications},
  Springer-Verlag, New York, 2014.

\bibitem{ReWe-SIAMJUQ21}
{\sc S.~Reich and S.~Weissmann}, {\em {F}okker--{P}lanck particle systems for
  {B}ayesian inference: Computational approaches}, SIAM/ASA J. Uncertain.
  Quantif., 9 (2021), pp.~446--482.

\bibitem{RiQuRiSc-SIAM24}
{\sc E.~Ribera~Borrell, J.~Quer, L.~Richter, and C.~Sch{\"u}tte}, {\em
  Improving control-based importance sampling strategies for metastable
  diffusions via adapted metadynamics}, SIAM Journal on Scientific Computing,
  46 (2024), pp.~S298--S323.

\bibitem{RiVo-arXiv22}
{\sc L.~Riou-Durand and J.~Vogrinc}, {\em Metropolis adjusted {L}angevin
  trajectories: a robust alternative to {H}amiltonian {M}onte {C}arlo}, arXiv
  preprint arXiv:2202.13230,  (2022).

\bibitem{RoRo-Book04}
{\sc G.~O. Roberts and J.~S. Rosenthal}, {\em General state space markov chains
  and {MCMC} algorithms}, Probability Surveys, 1 (2004), pp.~20--71.

\bibitem{RoSt-MCAP02}
{\sc G.~O. Roberts and O.~Stramer}, {\em Langevin diffusions and
  {Metropolis-Hastings} algorithms}, Methodology and Computing in Applied
  Probability, 4 (2002), pp.~337--357.

\bibitem{SuSyBoCa-NeurIPS22}
{\sc N.~Surjanovic, S.~Syed, A.~Bouchard-C{\^o}t{\'e}, and T.~Campbell}, {\em
  Parallel tempering with a variational reference}, in Advances in Neural
  Information Processing Systems (NeurIPS), 2022.

\bibitem{SwWa-PRL86}
{\sc R.~H. Swendsen and J.-S. Wang}, {\em Replica {M}onte {C}arlo simulation of
  spin-glasses}, Phys. Rev. Lett., 57 (1986), pp.~2607--2609.

\bibitem{SyBoDeDo-JRSSB22}
{\sc S.~Syed, A.~Bouchard-C{\^o}t{\'e}, G.~Deligiannidis, and A.~Doucet}, {\em
  Non-reversible parallel tempering: a scalable highly parallel {MCMC} scheme},
  J. R. Stat. Soc. Ser. B Stat. Methodol., 84 (2022), pp.~321--350.

\bibitem{TiPa-JRSSB18}
{\sc M.~K. Titsias and O.~Papaspiliopoulos}, {\em Auxiliary gradient-based
  sampling algorithms}, J. R. Stat. Soc. Ser. B Stat. Methodol., 80 (2018),
  pp.~749--767.

\bibitem{VeWi-NIPS19}
{\sc S.~Vempala and A.~Wibisono}, {\em Rapid convergence of the unadjusted
  {Langevin} algorithm: Isoperimetry suffices}, Advances in neural information
  processing systems, 32 (2019).

\bibitem{XiShLiByGi-SPL14}
{\sc T.~Xifara, C.~Sherlock, S.~Livingstone, S.~Byrne, and M.~Girolami}, {\em
  {Langevin} diffusions and the {Metropolis}-adjusted {Langevin} algorithm},
  Statistics \& Probability Letters, 91 (2014), pp.~14--19.

\bibitem{ZhCh-ICLR22}
{\sc Q.~Zhang and Y.~Chen}, {\em Path integral sampler: a stochastic control
  approach for sampling}, in International Conference on Learning
  Representations (ICLR), 2022.

\bibitem{Ziemer-Book12}
{\sc W.~P. Ziemer}, {\em Weakly Differentiable Functions: {Sobolev} Spaces and
  Functions of Bounded Variation}, Springer Science \& Business Media, 2012.

\end{thebibliography}

}

\end{document}